\documentclass[a4paper,12pt]{amsart}

\usepackage[margin=3truecm,top=4truecm,bottom=4truecm]{geometry}
\usepackage{amsthm,amsmath,amssymb,setspace,mleftright,constants,cite}
\usepackage[dvipdfmx]{graphicx,color}

\allowdisplaybreaks 

\usepackage[pagewise]{lineno}
\usepackage{etoolbox} 
\makeatletter 
\newcommand*\linenomathpatch{\@ifstar{\linenomathpatch@AMS}{\linenomathpatch@}}
\newcommand*\linenomathpatch@[1]{
	\expandafter\pretocmd\csname #1\endcsname {\linenomathWithnumbers}{}{}
	\expandafter\pretocmd\csname #1*\endcsname{\linenomathWithnumbers}{}{}
	\expandafter\apptocmd\csname end#1\endcsname {\endlinenomath}{}{}
	\expandafter\apptocmd\csname end#1*\endcsname{\endlinenomath}{}{}
}
\newcommand*\linenomathpatch@AMS[1]{
	\expandafter\pretocmd\csname #1\endcsname {\linenomathWithnumbersAMS}{}{}
	\expandafter\pretocmd\csname #1*\endcsname{\linenomathWithnumbersAMS}{}{}
	\expandafter\apptocmd\csname end#1\endcsname {\endlinenomath}{}{}
	\expandafter\apptocmd\csname end#1*\endcsname{\endlinenomath}{}{}
}
\let\linenomathWithnumbersAMS\linenomathWithnumbers
\patchcmd\linenomathWithnumbersAMS{\advance\postdisplaypenalty\linenopenalty}{}{}{}
\makeatother 

\linenomathpatch{equation}
\linenomathpatch*{gather}
\linenomathpatch*{multline}
\linenomathpatch*{align}
\linenomathpatch*{alignat}
\linenomathpatch*{flalign}

\makeatletter

\@addtoreset{equation}{section}
\makeatother

\newconstantfamily{c}{symbol=c}
\newconstantfamily{E}{symbol=\mathcal{E}}
\newconstantfamily{K}{symbol=\phi}

\theoremstyle{plain} 
\newtheorem{thm}{Theorem}[section]
\newtheorem{prop}[thm]{Proposition}
\newtheorem{cor}[thm]{Corollary}
\newtheorem{lem}[thm]{Lemma}

\theoremstyle{definition}

\theoremstyle{remark}

\newcommand{\N}{\mathbb{N}}
\newcommand{\R}{\mathbb{R}}
\newcommand{\Z}{\mathbb{Z}}

\renewcommand{\P}{\mathbb{P}}
\newcommand{\E}{\mathbb{E}}

\newcommand{\1}[1]{\mathbf{1}_{#1}}

\newcommand{\half}{\frac{1}{2}}

\renewcommand{\tilde}{\widetilde}

\renewcommand{\bar}{\overline}

\newcommand{\hyphen}{\textrm{-}}
\newcommand{\as}{\textrm{a.s.}}

\renewcommand{\S}{\mathbb{S}}
\newcommand{\bP}{\mathbf{P}}
\newcommand{\bE}{\mathbf{E}}

\begin{document}
\title[Lipschitz-type estimate for the frog model]{Lipschitz-type estimate for the frog model with Bernoulli initial configuration}

\author[V.~H.~CAN]{Van Hao CAN}
\address[V.~H.~CAN]{Institute of Mathematics, Vietnam Academy of Science and Technology, 18 Hoang Quoc Viet, Cau Giay, Hanoi, Vietnam.}
\email{cvhao@math.ac.vn}

\author[N.~KUBOTA]{Naoki KUBOTA}
\address[N.~KUBOTA]{College of Science and Technology, Nihon University, Chiba 274-8501, Japan.}
\email{kubota.naoki08@nihon-u.ac.jp}

\author[S.~NAKAJIMA]{Shuta NAKAJIMA}
\address[S.~NAKAJIMA]{Graduate School of Science and Technology, Meiji University, Kanagawa 214-8571, Japan.}
\email{njima@meiji.ac.jp}


\begin{abstract}
We consider the frog model with Bernoulli initial configuration, which is an interacting particle
system on the multidimensional lattice consisting of two states of particles: active and sleeping.
Active particles perform independent simple random walks.
On the other hand, although sleeping particles do not move at first, they become active and can move around when touched by active particles.
Initially, only the origin has one active particle, and the other sites have sleeping particles according to a Bernoulli distribution.
Then, starting from the original active particle, active ones are gradually generated and propagate across the lattice, with time.
It is of interest to know how the propagation of active particles behaves as the parameter of the Bernoulli distribution varies.
In this paper, we treat the so-called time constant describing the speed of propagation, and prove that the absolute difference between the time constants for parameters $p,q \in (0,1]$ is bounded from above and below by multiples of $|p-q|$. 
\end{abstract}

\maketitle

\section{Introduction}
In this paper, we consider the frog model with Bernoulli initial configuration on the $d$-dimensional lattice $\Z^d$ ($d \geq 2$).
This model consists of two types of particles: active and sleeping.
Active particles perform independent simple random walks on $\Z^d$.
On the other hand, although sleeping particles do not move at first, they become active and can move around when touched by active particles.
Initially, only the origin $0$ of $\Z^d$ has one active particle, and the sites of $\Z^d \setminus \{0\}$ have sleeping particles according to the Bernoulli distribution with some parameter.
Then, starting from the original active particle, active ones are gradually generated and propagate across $\Z^d$, with time.
Our main object of interest is the so-called time constant, which describes the speed of propagation of active frogs.
The time constant is a function of the parameter of the Bernoulli distribution, and this paper aims to investigate how the change in the parameter of the Bernoulli distribution affects the time constant.
In particular, we prove that the absolute difference between the time constants for parameters $p,q \in (0,1]$ is bounded from above and below by multiples of $|p-q|$.

\subsection{The model}\label{subsect:model}
Fix $d \geq 2$, and denote by $\mathcal{P}$ the set of all probability measures on $\N_0:=\N \cup \{0\}$ not concentrated in zero.
For a given $\Phi \in \mathcal{P}$, we consider a family $\omega=(\omega(x))_{x \in \Z^d}$ of independent random variables with the common law $\Phi$.
Furthermore, independently of $\omega$, let $\S=((S_k(x,\ell))_{k=0}^\infty)_{x \in \Z^d,\ell \in \N}$ be a family of independent simple random walks on $\Z^d$ satisfying that $S_0(x,\ell)=x$ for $x \in \Z^d$ and $\ell \in \N$.
For every $x,y \in \Z^d$, we now define the \emph{first passage time} $T(x,y)=T(x,y,\omega,\S)$ from $x$ to $y$ as
\begin{align*}
	T(x,y):=\inf\mleft\{ \sum_{i=0}^{m-1} \tau(x_i,x_{i+1}):
	\begin{minipage}{13.5em}
		$m \geq 1$ and $x_0,x_1,\dots,x_m \in \Z^d$\\
		with $x_0=x$ and $x_m=y$
	\end{minipage}
	\mright\},
\end{align*}
where
\begin{align*}
	\tau(x_i,x_{i+1})
	&= \tau\bigl( x_i,x_{i+1},\omega(x_i),S_\cdot(x_i,\cdot) \bigr)\\
	&:= \inf\bigl\{ k \geq 0:S_k(x_i,\ell)=x_{i+1} \text{ for some } \ell \in [1,\omega(x_i)] \bigr\},
\end{align*}
with the convention that $\tau(x_i,x_{i+1}):=\infty$ if $\omega(x_i)=0$.
Note that the first passage time satisfies the triangle inequality:
\begin{align}\label{eq:triangle}
	T(x,z) \leq T(x,y)+T(y,z),\qquad  x,y,z \in \Z^d.
\end{align}

On the event $\{ \omega(0) \geq 1 \}$, the first passage time $T(0,y)$ is intuitively interpreted as follows:
First, we place ``frogs'' on $\Z^d$ according to the initial configuration $\omega$, i.e., $\omega(x)$ frogs sit on each site $x$ (there is no frog at $x$ if $\omega(x)=0$).
In particular, the event $\{ \omega(0) \geq 1 \}$ guarantees that we assign at least one frog to the origin $0$.
The behavior of the $\ell$-th frog sitting on a site $x$ is controlled by the simple random walk $S_\cdot(x,\ell)$, but not all frogs move around from the beginning.
At first, the only frogs sitting on $0$ are active and perform simple random walks, whereas the other frogs are sleeping and do not move.
Each sleeping frog becomes active and starts to perform a simple random walk once it is touched by an original active frog.
When we repeat this procedure for the remaining sleeping and upcoming active frogs, $T(0,y)$ represents the minimum time at which an active frog reaches $y$.

Alves et al.~\cite[Section~2]{AlvMacPopRav01} proved that for a given $\Phi \in \mathcal{P}$, there exists a norm $\mu(\cdot)=\mu_\Phi(\cdot)$ on $\R^d$ (which is called the \emph{time constant}) such that almost surely on the event $\{ \omega(0) \geq 1 \}$,
\begin{align*}
	\lim_{n \to \infty} \frac{1}{n}T(0,nx)=\mu(x),\qquad x \in \Z^d.
\end{align*}
Furthermore, $\mu(\cdot)$ is invariant under permutations of the coordinates and under reflections in the coordinate hyperplanes, and satisfies
\begin{align}\label{eq:tc_bound}
	\|x\| \leq \mu(x) \leq \mu(\xi_1)\|x\|,\qquad x \in \R^d,
\end{align}
where $\|\cdot\|$ is the $\ell^1$-norm on $\R^d$ and $\xi_1$ is the first coordinate vector of $\R^d$.
As a consequence, Alves et al.~\cite[Theorem~{1.1}]{AlvMacPopRav01} also studied the asymptotic behavior of the random set
\begin{align*}
	\mathcal{B}(t):=\{x \in \Z^d:T(0,x) \leq t\},\qquad t \geq 0,
\end{align*}
which is the set of all sites visited by active frogs up to time $t$.
More precisely, for all $\epsilon>0$, the following holds almost surely on the event $\{ \omega(0) \geq 1 \}$:
For all large $t$,
\begin{align*}
	(1-\epsilon)t\mathcal{B} \cap \Z^d \subset \mathcal{B}(t) \subset (1+\epsilon)t\mathcal{B} \cap \Z^d,
\end{align*}
where $\mathcal{B}=\mathcal{B}_\Phi:=\{ x \in \R^d:\mu_\Phi(x) \leq 1 \}$, the \emph{asymptotic shape}.
This result is called the \emph{shape theorem} for the frog model, and it is clear from the properties of the time constant that the asymptotic shape $\mathcal{B}$ is a compact, convex, symmetric set with nonempty interior.

\subsection{Main result}\label{subsect:main}
Let $0<r \leq 1$ and denote by $\mathrm{Ber}(r)$ the Bernoulli distribution with parameter $r$.
We now consider the initial configuration $\omega$ governed by $\mathrm{Ber}(r)$:
\begin{align*}
	\P_r(\omega(x)=1)=1-\P_r(\omega(x)=0)=r,\qquad x \in \Z^d.
\end{align*}
In this setting, at most one frog is assigned to each site of $\Z^d$.
Hence, for abbreviation, write
\begin{align*}
	\S=(S_\cdot(x,1))_{x \in \Z^d}=(S_\cdot(x))_{x \in \Z^d}.
\end{align*}
For each $x \in \Z^d$, $P^x$ stands for the law of the simple random walk $S_\cdot(x)$ starting at $x$, and the product measure $P:=\prod_{x \in \Z^d}P^x$ is regarded as the law of frogs $\S$.
Then, writing $\bP_r:=\P_r \times P$ and $\mu_r(\cdot):=\mu_{\mathrm{Ber}(r)}(\cdot)$, we formulate the time constant and the asymptotic shape in the Bernoulli setting as follows:
$\mathbf{P}_r \hyphen \as$ on the event $\{ \omega(0)=1 \}$,
\begin{align}\label{eq:timeconst}
	\lim_{n \to \infty} \frac{1}{n}T(0,nx)=\mu_r(x),\qquad x \in \Z^d.
\end{align}

From \cite[Lemma~{2.2}]{Kub20}, $\mu_r(x)$ is decreasing in the parameter $r$ for every $x \in \R^d$.
This combined with \eqref{eq:tc_bound} implies that if $0<p<q \leq 1$, then for all $x \in \R^d \setminus \{0\}$,
\begin{align*}
	0 \leq \frac{\mu_p(x)-\mu_q(x)}{\|x\|} \leq \mu_p(\xi_1)-1.
\end{align*}
The aim of this paper is to give more precise upper and lower bounds for the above fraction, and the following theorem is our main result.

\begin{thm}\label{thm:Lipschitz}
Let $0<r_0<1$.
Then, there exist constants $A_1,A_2>0$ (which depend only on $d$ and $r_0$) such that if $r_0 \leq p<q \leq 1$, then for all $x \in \R^d \setminus \{0\}$,
\begin{align*}
	A_1(q-p) \leq \frac{\mu_p(x)-\mu_q(x)}{\|x\|} \leq A_2(q-p).
\end{align*}
In particular, for a fixed $x \in \R^d \setminus \{0\}$, the time constant $\mu_r(x)$ is strictly increasing in $r$ and Lipschitz continuous on every closed interval in $(0,1]$.
\end{thm}

The above Lipschitz-type estimate for the time constant is inherited to the asymptotic shape.
To explain this precisely, we introduce the Hausdorff distance between two subset of $\R^d$:
for any $\Gamma,\Gamma' \subset \R^d$,
\begin{align*}
	d_\mathcal{H}(A,B):=\inf\bigl\{ \delta>0:A \subset N_{\delta}(B) \text{ and } B \subset N_{\delta}(A) \bigr\},
\end{align*}
where for any $\Gamma \subset \R^d$,
\begin{align*}
	N_\delta(\Gamma):=\bigl\{ x \in \R^d:\|x-y\| \leq \delta \text{ for some } y \in \Gamma \bigr\}.
\end{align*}
Moreover, for any $r \in (0,1]$, set $\mathcal{B}_r:=\mathcal{B}_{\mathrm{Ber}(r)}$, which is the asymptotic shape for the parameter $r$ of the Bernoulli distribution.
Then, as a consequence of Theorem~\ref{thm:Lipschitz}, the next corollary gives the difference between two asymptotic shapes in the Hausdorff distance.

\begin{cor}\label{cor:Lipschitz_shape}
Let $0<r_0<1$.
Then, there exist constants $A'_1,A'_2>0$ (which depend only on $d$ and $r_0$) such that if $r_0 \leq p<q \leq 1$, then
\begin{align*}
	A'_1(q-p) \leq d_\mathcal{H}(\mathcal{B}_p,\mathcal{B}_q) \leq A'_2(q-p).
\end{align*}
In particular, the asymptotic shape $\mathcal{B}_r$ is strictly increasing in $r$.
\end{cor}

Let us finally comment on earlier works related to the above results.
In the frog model on $\Z^d$ ($d \geq 1$), \cite{AlvMacPop02} is the first published article investigating behaviors of the first passage time and the set of sites visited by active frogs.
That article treated the one-frog-per-site configuration and proved that the first passage time and the set of sites visited by active frogs grow linearly, with time.
In \cite{AlvMacPopRav01}, this result was extended to the case where the initial configuration of frogs is random, and we obtained the time constant and the asymptotic shape depending on the law of initial configuration of frogs, as stated at the end of Section~\ref{subsect:model}.
Although the aforementioned articles consider discrete-time simple random walks as frogs, Ram\'{\i}rez--Sidoravicius~\cite{RamSid04} independently derived the time constant and the asymptotic shape in the continuous-time frog model with the one-frog-per-site configuration (which consists of continuous-time simple random walks).

The propagation of active frogs has recently been investigated more closely in both discrete- and continuous-time settings, and we can find results on fluctuations and large deviations as follows:
In the discrete-time frog model on $\Z^d$ ($d \geq 2$) with random initial configuration, \cite{Kub19} provides some large deviation bounds for the first passage time.
More precisely, it proves that for any $\epsilon>0$, the conditional probabilities of $\{ T(0,x) \geq (1+\epsilon)\mu(x) \}$ and $\{ T(0,x) \leq (1-\epsilon)\mu(x) \}$ given $\{ \omega(0) \geq 1 \}$ decay at least stretched exponentially as $\|x\| \to \infty$.
Furthermore, if we consider the one-frog-per-site initial configuration, then \cite{CanNak19} gives a sublinear variance bound for the first passage time, i.e., for all $x \in \Z^d \setminus \{0\}$, the variance of $T(0,x)$ is bounded from above by a multiple of $\|x\|/\log\|x\|$.
On the other hand, in the continuous-time frog model on $\Z$, \cite{BerRam10} and \cite{ComQuaRam07,ComQuaRam09} prove the large deviation principle and the central limit theorem for the rightmost site visited by active frogs, respectively (We remark that these articles adopt a slightly different initial configuration from our frog model).

Although the aforementioned articles fix the law of initial configuration, it is also of interest to analyze the behavior of the time constant as the law of the initial configuration varies.
However, there are not a lot of results on this topic yet:
In \cite[Corollary~9]{JohJun18}, Johnsona and Junge introduced a (nontrivial) stochastic order for the initial configuration, and proved that the asymptotic shape is increasing in that stochastic order.
In \cite{Kub20}, the second author showed continuities for the time constant and the asymptotic shape in the law of the initial configuration.
Note that Theorem~\ref{thm:Lipschitz} and Corollary~\ref{cor:Lipschitz_shape} strength the continuities as long as the initial configuration is Bernoulli-distributed.
We may be able to generalize these results to other classes of distributions.
However, \cite[Theorem~{1.2}]{AlvMacPopRav01} shows that there exists a class of distributions such that we cannot obtain similar lower bounds to those in Theorem~\ref{thm:Lipschitz} and Corollary~\ref{cor:Lipschitz_shape} as follows:
Suppose that $\Phi \in \mathcal{P}$ satisfies
\begin{align}\label{eq:heavytail}
	\Phi([t,\infty)) \geq (\log t)^{-\delta} \quad \text{for some $\delta \in (0,d)$ and for all large $t$}.
\end{align}
Then, $\mu_\Phi(\cdot)$ and $\mathcal{B}_\Phi$ coincide with the $\ell^1$-norm on $\R^d$ and the closed $\ell^1$-unit ball, respectively.
This implies that if $\Phi,\Psi \in \mathcal{P}$ satisfy \eqref{eq:heavytail}, then $\mu_\Phi(x)-\mu_\Psi(x)=0$ for all $x \in \R^d$ and $d_\mathcal{H}(\mathcal{B}_\Phi,\mathcal{B}_\Psi)=0$, and lower bounds as in Theorem~\ref{thm:Lipschitz} and Corollary~\ref{cor:Lipschitz_shape} never hold.
In this way, there is a possibility that the propagation of active frogs exhibits unusual behavior if the initial configuration has a heavy-tailed distribution.
In fact, \cite[Proposition~{1.4}]{Kub20} exhibits that \eqref{eq:heavytail} can collapse the continuity of the time constant in the law of the initial configuration.
Moreover, \cite{BezDPerKru21} and \cite{BezDPerKuc24} show that in the continuous-time frog model on $\Z$ with random initial configuration, the set of sites visited by active frogs becomes infinite in a finite time if the law of initial configuration satisfies a heavy-tail condition weaker than \eqref{eq:heavytail}.

\subsection{Organization of the paper}
Let us describe how the present article is organized.
In Section~\ref{sect:prelim}, we prepare some estimates for the first passage time which are used throughout the paper.
Furthermore, our proofs are often done by constructing a finitely dependent random variables and estimating the upper tail of their sum.
To do this, the result obtained by Can and Nakajima~\cite[Lemma~{2.6}]{CanNak19} is useful, and we state their result as Proposition~\ref{prop:CN} below for the convenience of the reader.

Section~\ref{sect:pf_main} is devoted to the proof of Theorem~\ref{thm:Lipschitz} and Corollary~\ref{cor:Lipschitz_shape}.
The main tool is Russo's formula (see for instance \cite[Theorem~{2.32}]{Gri99_book}) for the first passage time.
This tells us that Theorem~\ref{thm:Lipschitz} follows by estimating the influence of the absence of one frog on the propagation of active frogs.
Although Propositions~\ref{prop:influ_upper} and \ref{prop:influ_lower} below play a role in estimating this effect, we use these propositions without their proofs and complete the proofs of Theorem~\ref{thm:Lipschitz} and Corollary~\ref{cor:Lipschitz_shape} for now.

Section~\ref{sect:influence} deals with the proofs of Propositions~\ref{prop:influ_upper} and \ref{prop:influ_lower}, which provide upper and lower bounds for the influence of the absence of one frog on the propagation of active frogs, respectively.
In particular, to prove Proposition~\ref{prop:influ_lower}, we count the number of frogs which contribute but delay the propagation of active frogs.
Proposition~\ref{prop:delay} estimates that number of frogs from below.
Its proof is postponed into Section~\ref{sect:delay} since we need some more work.

In Section~\ref{sect:delay}, we prove Proposition~\ref{prop:delay}.
The key to proving Proposition~\ref{prop:influ_lower} is to observe that the first passage time is strictly greater than the $\ell^1$-norm with high probability.
In \cite[Lemma~{5.2}]{Kub20}, this behavior of the first passage time has been already observed under the law $\P_r$ with a sufficiently small parameter $r$.
However, The result obtained in \cite[Lemma~{5.2}]{Kub20} is not enough to prove Proposition~\ref{prop:delay} since $r$ is not necessarily small in the present paper.
Hence, we first solve this problem in Section~\ref{subsect:Nakajima}, and next show Proposition~\ref{prop:delay} in Section~\ref{subsect:pf_delay}.

We close this section with some general notation.
Let $r_0$ be a fixed parameter in $(0,1)$, and this information is dropped from all statements below.
Denote by $E$, $\E_r$ and $\bE_r$ the expectation associated to the laws $\P_r$, $P$ and $\bP_r$ stated in Section~\ref{subsect:main}, respectively.
The $\ell^1$-norm on $\R^d$ is designated by $\|\cdot\|$, and set $B(x,R):=\{ y \in \Z^d:\|y-x\| \leq R \}$ for $x \in \Z^d$ and $R \geq 0$.
Moreover, for each $i \in [1,d]$, write $\xi_i$ for the $i$-th coordinate vector of $\R^d$.
Finally, throughout the paper, we use $c$ and $c'$ to denote arbitrary positive constants, which may change from line to line.

\section{Preliminaries}\label{sect:prelim}
In this section, we prepare some notation and lemmata which are often used throughout this paper.
First of all, to make our statements simple, let us introduce a class of rapidly decreasing functions on $[0,\infty)$.
Denote by $\mathcal{K}(d,r_0)$ the set of all functions $\phi$ on $[0,\infty)$ with the form
\begin{align*}
	\phi(t)=ce^{-c't^\alpha},\qquad t \geq 0,
\end{align*}
where $c$, $c'$ and $\alpha$ are some positive constants depending only on $d$ and $r_0$.
Each element of $\mathcal{K}(d,r_0)$ is called the stretched exponential (or the Kohlrausch--Williams--Watts) function.
Note that every $\phi \in \mathcal{K}(d,r_0)$ decays faster than all positive power functions, i.e., $\lim_{t \to \infty} t^\beta \phi(t)=0$ holds for all $\beta>0$.

We next state some upper tail estimates for the first passage time.
For any $\omega \in \{ 0,1 \}^{\Z^d}$, $z \in \Z^d$ and $s \in \{ 0,1 \}$, let $\omega_z^s$ be the initial configuration $\omega$ with $\omega(z)$ forced to take the value $s$, i.e.,
\begin{align*}
	\omega_z^s(x):=
	\begin{cases}
		s, & \text{if } x=z,\\
		\omega(x), & \text{if } x \not= z.
	\end{cases}
\end{align*}
In particular, set $\langle \omega \rangle:=\omega_0^1$ for any $\omega \in \{ 0,1 \}^{\Z^d}$.
Moreover, define for $x,y \in \Z^d$,
\begin{align*}
	T_z^s(x,y):=T(x,y,\omega_z^s,\S),
\end{align*}
which is the first passage time from $x$ to $y$ in the initial configuration $\omega_z^s$.
The next lemma says that we can control the tail probability of the first passage time uniformly in the parameter $r \in [r_0,1]$ and the mandatory absence of frogs.

\begin{lem}\label{lem:AP}
There exist $\phi_0 \in \mathcal{K}(d,r_0)$ and a constant $C_0 \geq 1$ (which depends only on $d$ and $r_0$) such that for all $y \in \Z^d$ and $t \geq C_0\|y\|$,
\begin{align*}
	\sup_{\substack{r_0 \leq r \leq 1\\ z \in \Z^d \setminus \{0\}}} \bP_r\mleft( T_z^0(0,y) \geq t \middle| \omega(0)=1 \mright)
	\leq \phi_0(t).
\end{align*}
In particular, for any $\beta>0$,
\begin{align*}
	\sup_{\substack{r_0 \leq r \leq 1\\ y,z \in \Z^d \setminus \{0\}}} \frac{1}{\|y\|^{\beta}}\,\bE_r\mleft[ T_z^0(0,y)^\beta \middle| \omega(0)=1 \mright]
	<\infty.
\end{align*}
\end{lem}
\begin{proof}
Let $r \in [r_0,1]$ and $z \in \Z^d \setminus \{0\}$.
Note that $T(0,y,\langle \omega_z^0 \rangle,\S)$ is independent of both $\omega(0)$ and $\omega(z)$.
Hence, for all $y \in \Z^d$ and $t \geq 0$,
\begin{align*}
	\bP_r\mleft( T_z^0(0,y) \geq t \middle| \omega(0)=1 \mright)
	= \bP_r\mleft( T(0,y,\langle \omega_z^0 \rangle,\S) \geq t \mright).
\end{align*}
In addition, due to $r \geq r_0$, a coupling of initial configurations (see \cite[Lemma~{2.2}]{Kub20}) proves
\begin{align*}
	\bP_r\mleft( T(0,y,\langle \omega_z^0 \rangle,\S) \geq t \mright)
	\leq \bP_{r_0}\mleft( T(0,y,\langle \omega_z^0 \rangle,\S) \geq t \mright).
\end{align*}
By using the independence, the right side above is equal to
\begin{align*}
	&\bP_{r_0}\mleft( T(0,y,\langle \omega_z^0 \rangle,\S) \geq t,\,\omega(0)=1,\,\omega(z)=0 \mright) \times \frac{1}{r_0(1-r_0)}\\
	&\leq (1-r_0)^{-1}\bP_{r_0}(T(0,y) \geq t|\omega(0)=1).
\end{align*}
With these observations, for all $y \in \Z^d$ and $t \geq 0$,
\begin{align*}
	\bP_r\mleft( T_z^0(0,y) \geq t \middle| \omega(0)=1 \mright)
	\leq (1-r_0)^{-1}\bP_{r_0}( T(0,y) \geq t|\omega(0)=1).
\end{align*}
To estimate the last probability, we use the following result proved in \cite[Proposition~{2.4}]{Kub19}:
there exist $\phi \in \mathcal{K}(d,r_0)$ and a constant $C_0 \geq 1$ (which depends only on $d$ and $r_0$) such that for all $y \in \Z^d$ and $t \geq C_0\|y\|$,
\begin{align*}
	\bP_{r_0}( T(0,y) \geq t|\omega(0)=1)
	\leq \phi(t).
\end{align*}
Therefore, for all $y \in \Z^d$ and $t \geq C_0\|y\|$,
\begin{align*}
	\bP_r\mleft( T_z^0(0,y) \geq t \middle| \omega(0)=1 \mright)
	\leq (1-r_0)^{-1}\phi(t),
\end{align*}
which implies the first assertion of the lemma since $r$ and $z$ are arbitrary and $\phi \in \mathcal{K}(d,r_0)$.
The second assertion of the lemma is a direct consequence of the first one.
\end{proof}

For any $y \in \Z^d \setminus \{0\}$, write $\mathcal{O}(y)=\mathcal{O}(y,\omega,\S)$ for the set of all sequences $(x_i)_{i=0}^m$ of distinct points in $\Z^d$ satisfying that $T(0,y)=\sum_{i=0}^{m-1}\tau(x_i,x_{i+1})$.
Each element of $\mathcal{O}(y)$ represents an optimal selection and order of frogs to realize the first passage time from $0$ to $y$.
We choose and fix a certain element of $\mathcal{O}(y)$ with a deterministic rule, and denote it by $\gamma_\mathcal{O}(y)=\gamma_\mathcal{O}(y,\omega,\S)$.
In addition, for each $z \in \gamma_\mathcal{O}(y)$, let $\bar{z}$ be the next points of $z$ in $\gamma_\mathcal{O}(y)$ (if any).
Then, the next lemma guarantees that most of the frogs in $\gamma_\mathcal{O}(y)$ do not take so long to pass the baton to their next candidates.
The proof is analogous to those of \cite[Lemmata~{4.3} and 4.4]{Kub20}, so we omit it.

\begin{lem}\label{lem:next}
There exists $\Cl[K]{phi:next} \in \mathcal{K}(d,r_0)$ such that for all $t \geq 0$,
\begin{align*}
	\sup_{\substack{r_0 \leq r \leq 1\\ y \in \Z^d \setminus \{0\}}}
	\bP_r\bigl( T(0,y)=\tau(0,y) \geq t \big| \omega(0)=1 \bigr)
	\leq \Cr{phi:next}(t).
\end{align*}
In particular, there exists $\Cl[K]{phi:range} \in \mathcal{K}(d,r_0)$ such that for all $y \in \Z^d \setminus \{0\}$,
\begin{align*}
	\sup_{r_0 \leq r \leq 1}
	\bP_r\Bigl( \tau(z,\bar{z}) \geq \|y\|^{1/2} \text{ for some } z \in \gamma_\mathcal{O}(y) \setminus \{y\}
	\Big| \omega(0)=1 \Bigr)
	\leq \Cr{phi:range}(\|y\|).
\end{align*}
\end{lem}

We finally mention a result for sums of dependent random variables.
For a given $L \in \N$, a family $(X(v))_{v \in \Z^d}$ of random variables is said to be \emph{$L$-dependent} if any two sub-families $(X(u))_{u \in A}$ and $(X(v))_{v \in B}$ are independent whenever $A,B \subset \Z^d$ satisfy that $\|u-v\|>L$ for all $u \in A$ and $v \in B$.
Moreover, for each $M \geq 1$, denote by $\Pi_M$ the set of all sequences $(x_i)_{i=0}^m$ of distinct points in $\Z^d$ satisfying that $\sum_{i=0}^{m-1}\|x_i-x_{i+1}\| \leq M$.
Then, Can and Nakajima~\cite[Lemma~{2.6}]{CanNak19} obtained the following upper tail estimate for sums of an $L$-dependent family of random variables taking values in $\{0,1\}$.

\begin{prop}\label{prop:CN}
Let $(X(v))_{v \in \Z^d}$ be a family of random variables taking values in $\{0,1\}$, defined on a measurable space equipped with a probability measure $\bP$.
Furthermore, assume that $(X(v))_{v \in \Z^d}$ is $L$-dependent and satisfies
\begin{align*}
	p:=\sup_{v \in \Z^d}\bP(X(v)=1) \leq (3L+1)^{-d}.
\end{align*}
Then, there exists a constant $\Cl{CN} \geq 1$ (which depends only on $d$) such that for any $M \geq 1$ and $t \geq \Cr{CN}L^{d+1}\max\{ 1,Mp^{1/d} \}$,
\begin{align*}
	\bP\biggl( \max_{\pi \in \Pi_M} \sum_{v \in \pi} X(v) \geq t \biggr)
	\leq 2^d\exp\biggl\{ -\frac{t}{(16L)^d} \biggr\}.
\end{align*}
\end{prop}

\section{Proof of the main result}\label{sect:pf_main}
In this section, we prove Theorem~\ref{thm:Lipschitz} and Corollary~\ref{cor:Lipschitz_shape}.
The main tool here is Russo's formula, which computes the derivative of $\E_r[X]$ (as a function of $r$) for a random variable $X$ which depends only on the initial configuration in a finite region (see for instance \cite[Theorem~{2.32}]{Gri99_book}).
The basic idea of the proof of Theorem~\ref{thm:Lipschitz} is to apply Russo's formula to $E[T(0,y)]$.
However, this approach does not work directly because $E[T(0,y)]$ may depend on the initial configuration in the whole of $\Z^d$.
To overcome this problem, let us define for $s \in \{0,1\}$, $y,z \in \Z^d$ and $N \in \N$,
\begin{align*}
	U_z^s(y,N):=T(0,y,\langle \omega_z^s \rangle,\S)\1{\{ T(0,y,\langle \omega_z^s \rangle,\S) \leq N \}}.
\end{align*}
It is clear that $U_z^s(y,N)$ is independent of the initial configuration outside $B(0,N)$.
First, observe the behavior of $\bE_r[U_0^1(y,N)]$ as $N$ and $\|y\|$ tend to infinity.

\begin{lem}\label{lem:U_lim}
We have for any $r \in [r_0,1]$ and $x \in \R^d \setminus \{0\}$,
\begin{align*}
	\lim_{k \to \infty}\lim_{N \to \infty} \frac{1}{k}\bE_r[U_0^1(kx,N)]=\mu_r(x).
\end{align*}
\end{lem}
\begin{proof}
Fix $r \in [r_0,1]$ and $x \in \R^d \setminus \{0\}$.
The monotone convergence theorem, combined with Lemma~\ref{lem:AP}, implies that for any $k \geq 1$,
\begin{align*}
	\lim_{N \to \infty} \frac{1}{k}\bE_r[U_0^1(kx,N)]
	&= \lim_{N \to \infty} \frac{1}{k}\bE_r\mleft[ T_0^1(0,kx)\1{\{ T_0^1(0,kx) \leq N \}} \mright]\\
	&= \frac{1}{k}\bE_r[T_0^1(0,kx)]
		= \frac{1}{k}\bE_r[T(0,kx)|\omega(0)=1].
\end{align*}
Once we prove that $(T(0,kx)/k)_{k=1}^\infty$ is uniformly integrable under $\bP_r(\cdot|\omega(0)=1)$, the rightmost side above converges to $\mu_r(x)$ as $k \to \infty$ by \eqref{eq:timeconst} and the Vitali convergence theorem, and the lemma follows.

It remains to check the uniform integrability of $(T(0,kx)/k)_{k=1}^\infty$.
The Cauchy--Schwarz inequality and Lemma~\ref{lem:AP} imply that there exists a constant $c$ (which depends only on $d$ and $r_0$) such that for any $\lambda \geq C_0\|x\|$,
\begin{align*}
	&\sup_{k \geq 1}\bE_r\biggl[ \frac{1}{k}T(0,kx) \1{\{ T(0,kx)/k \geq \lambda \}} \bigg| \omega(0)=1 \biggr]\\
	&\leq \sup_{k \geq 1}\frac{1}{k} \bE_r[T(0,kx)^2|\omega(0)=1]^{1/2}
		\,\bP_r(T(0,kx) \geq \lambda k|\omega(0)=1)^{1/2}\\
	&\leq c\|x\|\sup_{k \geq 1}\phi_0(\lambda k)^{1/2}
		\leq c\|x\|\phi_0(\lambda)^{1/2}.
\end{align*}
The rightmost side converges to zero as $\lambda \to \infty$, and hence $(T(0,kx)/k)_{k=1}^\infty$ is uniformly integrable under $\bP_r(\cdot|\omega(0)=1)$.
\end{proof}

Although Lemma~\ref{lem:U_lim} treats the behavior of $\bE_r[U_0^1(y,N)]$, let us next estimate the difference between $T_z^s(0,y)$ and $U_z^s(y,N)$ for $s \in \{0,1\}$ and $z \in \Z^d \setminus \{0\}$.

\begin{lem}\label{lem:comparable}
There exists $\Cl[K]{phi:comp} \in \mathcal{K}(d,r_0)$ such that for all $s \in \{0,1\}$, $y,z \in \Z^d \setminus \{0\}$, $N \geq C_0\|y\|$ and $r \in [r_0,1]$,
\begin{align}\label{eq:comparable}
	0 \leq \bE_r\bigl[ T_z^s(0,y)-U_z^s(y,N) \big| \omega(0)=1 \bigr] \leq \Cr{phi:comp}(N)\|y\|.
\end{align}
\end{lem}
\begin{proof}
Fix $s \in \{0,1\}$, $y,z \in \Z^d \setminus \{0\}$, $N \geq C_0\|y\|$ and $r \in [r_0,1]$.
Note that on the event $\{ \omega(0)=1 \}$,
\begin{align*}
	T_z^s(0,y)-U_z^s(y,N)=T_z^s(0,y)\1{\{ T_z^s(0,y)>N \}}.
\end{align*}
Hence, the first inequality of \eqref{eq:comparable} is trivial.
For the second inequality of \eqref{eq:comparable}, we use the above equality and the Cauchy--Schwarz inequality:
\begin{align*}
	&\bE_r\bigl[ T_z^s(0,y)-U_z^s(y,N) \big| \omega(0)=1 \bigr]\\
	&= \bE_r\mleft[ T_z^s(0,y)\1{\{ T_z^s(0,y)>N \}} \middle| \omega(0)=1 \mright]\\
	&\leq \bE_r\mleft[ T_z^s(0,y)^2 \middle| \omega(0)=1 \mright]^{1/2}
		\,\bP_r\mleft( T_z^s(0,y)>N \middle| \omega(0)=1 \mright)^{1/2}.
\end{align*}
The fact that $T_z^0(0,y) \geq T_z^s(0,y)$ and Lemma~\ref{lem:AP} imply that there exists a constant $c$ (which depends only on $d$ and $r_0$) such that the rightmost side above is smaller than or equal to
\begin{align*}
	&\bE_r\mleft[ T_z^0(0,y)^2 \middle| \omega(0)=1 \mright]^{1/2}
		\,\bP_r\mleft( T_z^0(0,y)>N \middle| \omega(0)=1 \mright)^{1/2}\\
	&\leq c\phi_0(N)^{1/2}\|y\|.
\end{align*}
Therefore, due to $\phi_0 \in \mathcal{K}(d,r_0)$, the second inequality of \eqref{eq:comparable} follows by taking $\Cr{phi:comp}(\cdot):=c\phi_0(\cdot)^{1/2}$.
\end{proof}

We are now in a position to prove Theorem~\ref{thm:Lipschitz}.

\begin{proof}[\bf Proof of Theorem~\ref{thm:Lipschitz}]
It suffices to show that there exist constants $A_1,A_2>0$ (which depend only on $d$ and $r_0$) such that if $\|y\|$ and $N \geq C_0\|y\|$ are large enough, then
\begin{align}\label{eq:derivative}
	A_1\|y\| \leq -\frac{d}{dr}\bE_r[U_0^1(y,N)] \leq A_2\|y\|,\qquad r_0 \leq r \leq 1.
\end{align}
Indeed, \eqref{eq:derivative} yields that the functions 
\begin{align*}
	\bE_r[U_0^1(y,N)]+A_ir\|y\|,\qquad i=1,2,
\end{align*}
are decreasing and increasing in $r \in [r_0,1]$, respectively.
Hence, for any $p,q \in [r_0,1]$ with $p<q$ and $x \in \Z^d \setminus \{0\}$, if $k \geq 1$ and $N \geq C_0\|kx\|$ are large enough, then
\begin{align*}
	A_1(q-p)\|x\|
	\leq \frac{1}{k}\bE_p[U_0^1(kx,N)]-\frac{1}{k}\bE_q[U_0^1(kx,N)]
	\leq A_2(q-p)\|x\|.
\end{align*}
From Lemma~\ref{lem:U_lim}, letting $N \to \infty$ and $k \to \infty$ proves that for all $p,q \in [r_0,1]$ with $p<q$ and $x \in \Z^d \setminus \{0\}$,
\begin{align*}
	A_1(q-p)\|x\| \leq \mu_p(x)-\mu_q(x) \leq A_2(q-p)\|x\|.
\end{align*}
This compound inequality can easily be extended to all $x \in \R^d \setminus \{0\}$ since $\mu_p(\cdot)$ and $\mu_q(\cdot)$ are norms on $\R^d$, and the theorem follows (the strict monotonicity and the Lipschitz continuity of the time constant are direct consequences of the first assertion since we can take $r_0 \in (0,1)$ arbitrarily).

To prove \eqref{eq:derivative}, let us prepare the following propositions, which provide upper and lower bounds for sums of influences on the first passage time.

\begin{prop}\label{prop:influ_upper}
There exists a constant $\Cl{influ_upper}>0$ (which depends only on $d$ and $r_0$) such that if $\|y\|$ is large enough and $N \geq C_0\|y\|$, then for all $r \in [r_0,1]$,
\begin{align*}
	\sum_{z \in B(0,N) \setminus \{0,y\}}
	\bE_r\mleft[ T_z^0(0,y)-T_z^1(0,y) \middle| \omega(0)=1 \mright]
	\leq \Cr{influ_upper}\|y\|.
\end{align*}
\end{prop}

\begin{prop}\label{prop:influ_lower}
There exists a constant $\Cl{influ_lower}>0$ (which depends only on $d$ and $r_0$) such that if $\|y\|$ is large enough and $N \geq C_0\|y\|$, then for all $r \in [r_0,1]$,
\begin{align*}
	\sum_{z \in B(0,N) \setminus \{0,y\}}
	\bE_r\mleft[ T_z^0(0,y)-T_z^1(0,y) \middle| \omega(0)=1 \mright]
	\geq \Cr{influ_lower}\|y\|.
\end{align*}
\end{prop}

The proofs of these propositions are postponed until Sections~\ref{subsect:influ_upper} and \ref{subsect:influ_lower}, and we continue with the proof of \eqref{eq:derivative}.
Fix $y \in \Z^d \setminus \{0\}$ and $N \geq C_0\|y\|$.
Since $E[U_0^1(y,N)]$ depends only on the initial configuration in $B(0,N)$, we can use Russo's formula to obtain for all $r \in [r_0,1]$,
\begin{align*}
	-\frac{d}{dr}\bE_r[U_0^1(y,N)]
	&= \sum_{z \in \Z^d} \E_r\bigl[ E[U_z^0(y,N)]-E[U_z^1(y,N)] \bigr]\\
	&= \sum_{z \in B(0,N) \setminus \{0,y\}} \bE_r\mleft[ U_z^0(y,N)-U_z^1(y,N) \mright].
\end{align*}
Here we used the fact that $U_z^0(y,N)=U_z^1(y,N)$ holds for all $z \in B(0,N)^c \cup \{0,y\}$ in the last equality.
On the other hand, Lemma~\ref{lem:comparable} tells us that for each $r \in [r_0,1]$,
\begin{align*}
	&\sum_{z \in B(0,N) \setminus \{0,y\}}
		\Bigl| \bE_r\mleft[ T_z^0(0,y)-T_z^1(0,y) \middle| \omega(0)=1 \mright]-\bE_r\mleft[ U_z^0(y,N)-U_z^1(y,N) \mright] \Bigr|\\
	&\leq 2(2N+1)^d\Cr{phi:comp}(N)\|y\|.
\end{align*}
We now assume that $\|y\|$ and $N \geq C_0\|y\|$ are large enough to establish Propositions~\ref{prop:influ_upper} and \ref{prop:influ_lower} and to have $2(2N+1)^d\Cr{phi:comp}(N) \leq \Cr{influ_lower}/2$, respectively.
Then, the above observations, combined with Propositions~\ref{prop:influ_upper} and \ref{prop:influ_lower}, imply that for all $r \in [r_0,1]$,
\begin{align*}
	\frac{\Cr{influ_lower}}{2}\|y\|
	\leq -\frac{d}{dr}\bE_r[U_0^1(y,N)]
	\leq (\Cr{influ_upper}+\Cr{influ_lower})\|y\|,
\end{align*}
and \eqref{eq:derivative} follows by taking $A_1:=\Cr{influ_lower}/2$ and $A_2:=\Cr{influ_upper}+\Cr{influ_lower}$.
\end{proof}

Corollary~\ref{cor:Lipschitz_shape} is a direct consequence of Theorem~\ref{thm:Lipschitz} and its proof is completely the same as those of \cite[Corollary~{1.2}]{Dem21} and \cite[Corollary~3]{KubTak22}.
However, we give the proof of the corollary for completeness since it is not so long.

\begin{proof}[\bf Proof of Corollary~\ref{cor:Lipschitz_shape}]
Assume that $r_0 \leq p<q \leq 1$.
Note that
\begin{align*}
	d_\mathcal{H}(\mathcal{B}_p,\mathcal{B}_q)=\inf\{ \delta>0 : \text{$\mathcal{B}_q \subset  N_{\delta} (\mathcal{B}_p)$} \},
\end{align*}
since $\mathcal{B}_p \subset \mathcal{B}_q$ holds by the monotonicity of the time constant in the parameter of the Bernoulli distribution (see below \eqref{eq:timeconst}).
Set $\delta_-:=A_1(q-p)/(2\mu_{r_0}(\xi_1)^2)$.
Then, \eqref{eq:tc_bound}, Theorem~\ref{thm:Lipschitz} and the monotonicity imply that for all $y \in \R^d$ with $\mu_p(y)=1$ and $x \in B(y,\delta_-)$,
\begin{align*}
	\mu_q(x)
	&\leq \mu_q(x-y)+\mu_q(y)
		\leq \mu_{r_0}(\xi_1)\delta_-+\mu_p(y)-A_1(q-p)\|y\|\\
	&\leq 1+\frac{A_1}{2}(q-p)\frac{1}{\mu_{r_0}(\xi_1)}-A_1(q-p)\frac{1}{\mu_p(\xi_1)}\\
	&\leq 1-\frac{A_1}{2\mu_{r_0}(\xi_1)}(q-p)<1.
\end{align*}
This implies that the boundary of $\mathcal{B}_q$ is not included in $N_\delta(\mathcal{B}_p)$ for all $\delta \in (0,\delta_0]$, and hence
\begin{align*}
	d_\mathcal{H}(\mathcal{B}_p, \mathcal{B}_q)
	\geq \delta_-=\frac{A_1}{2\mu_{r_0}(\xi_1)^2}(q-p).
\end{align*}
This is the desired lower bound for the Hausdorff distance since $A_1$ and $\mu_{r_0}(\xi_1)$ depend only on $d$ and $r_0$.
On the other hand, for any $z \in \R^d$ with $\mu_q(z)=1$, let $\xi_z:=z/\|z\|$ and $y_z:=\xi_z/\mu_p(\xi_z)$.
Note that $\|\xi_z\|=1$ and $\mu_p(y_z)=1$ (and $y_z \in \mathcal{B}_p$ as well).
Hence, Theorem~\ref{thm:Lipschitz} proves that
\begin{align*}
	\|y_z-z\|
	&= \biggl\| \frac{\xi_z}{\mu_p(\xi_z)}\mu_p(y_z)-\frac{\xi_z}{\mu_q(\xi_z)}\mu_q(z) \biggr\|\\
	&\leq \sup_{\|\xi\|=1}\frac{\mu_p(\xi)-\mu_q(\xi)}{\mu_p(\xi)\mu_q(\xi)}
		\leq A_2(q-p)=:\delta_+.
\end{align*}
This means that for any $z \in \R^d$ with $\mu_q(z)=1$, there exists $y \in \mathcal{B}_p$ such that $\|y-z\| \leq \delta_+$, and it is easy to see that $\mathcal{B}_q \subset N_{\delta_+}(\mathcal{B}_p)$ holds.
Therefore, we obtain
\begin{align*}
	d_{\mathcal{H}}(\mathcal{B}_p,\mathcal{B}_q) \leq \delta_+=A_2(q-p),
\end{align*}
and the desired upper bound for the Hausdorff distance follows since $A_2$ also depends only on $d$ and $r_0$.
The strict monotonicity of the asymptotic shape is a direct consequence of the first assertion since $r_0 \in (0,1)$ is arbitrary, and the proof is complete.
\end{proof}

\section{Upper and lower bounds for sums of influences}\label{sect:influence}
In this section, we prove Propositions~\ref{prop:influ_upper} and \ref{prop:influ_lower}.
The task is here to check that if $\|y\|$ and $N$ are sufficiently large, then for any $r \in [r_0,1]$, the sum of influences
\begin{align*}
	\sum_{z \in B(0,N) \setminus \{0,y\}} \bE_r\mleft[ T_z^0(0,y)-T_z^1(0,y) \middle| \omega(0)=1 \mright]
\end{align*}
is bounded from above and below by multiples of $\|y\|$.
To do this, we interpret the quantity $T_z^0(0,y)-T_z^1(0,y)$ as the delay time caused by the absence of the frog on site $z$.
Intuitively, it is expected that one frog does not have a great impact on the propagation of active frogs, and this suggests that the sum of influences is controllable from above.
In Section~\ref{subsect:influ_upper}, we actually carry out this strategy and derive the desired upper bound for the sum of influences (Propositions~\ref{prop:influ_upper}).
On the other hand, it is easily seen that $T_z^0(0,y)-T_z^1(0,y) \geq 1$ holds if the frog sitting on site $z$ decisively contributes to the first passage time from $0$ to $y$.
Hence, in Section~\ref{subsect:influ_lower}, we count the number of such frogs and estimate the sum of influences from below (Proposition~\ref{prop:influ_lower}).

\subsection{Proof of Proposition~\ref{prop:influ_upper}}\label{subsect:influ_upper}
We begin by preparing some notation and lemmata.
Recall from above Lemma~\ref{lem:next} that for each $y \in \Z^d \setminus \{0\}$, $\gamma_\mathcal{O}(y)$ is a specific element of $\mathcal{O}(y)$ chosen with a deterministic rule, and for each $z \in \gamma_\mathcal{O}(y)$, $\bar{z}$ is the next point of $z$ in $\gamma_\mathcal{O}(y)$ (if any).
In addition, let $\underline{z}$ denote the previous point of $z$ in $\gamma_\mathcal{O}(y)$ (if any).
When $z \in \tilde{\gamma}_\mathcal{O}(y):=\gamma_\mathcal{O}(y) \setminus \{0,y\}$, both $\underline{z}$ and $\overline{z}$ actually exist and we can consider the quantity $T_z^0(\underline{z},\overline{z})$ (which is the first passage time from $\underline{z}$ to $\overline{z}$ in the case where the use of the frog sitting on $z$ is prohibited).
The following lemmata provide some estimates for $T_z^0(\underline{z},\overline{z})$, and those play a key role in the proof of Proposition~\ref{prop:influ_upper}.

\begin{lem}\label{lem:skip}
There exists a constant $\Cl{skip}>0$ (which depends only on $d$ and $r_0$) such that for all $y,z \in \Z^d \setminus \{0\}$,
\begin{align*}
	\sup_{r_0 \leq r \leq 1} \bE_r\mleft[ T_z^0(\underline{z},\overline{z})^2 \1{\{ z \in \tilde{\gamma}_\mathcal{O}(y) \}} \middle| \omega(0)=1 \mright]
	\leq \Cr{skip}.
\end{align*}
\end{lem}

\begin{lem}\label{lem:skip_negl}
There exists $\Cl[K]{phi:skip} \in \mathcal{K}(d,r_0)$ such that for all $y \in \Z^d \setminus \{0\}$,
\begin{align*}
	\sup_{r_0 \leq r \leq 1}\bE_r\biggl[ \1{\mathcal{E}(y)^c} \sum_{z \in \tilde{\gamma}_\mathcal{O}(y)}
	T_z^0(\underline{z},\overline{z}) \bigg| \omega(0)=1 \biggr]
	\leq \Cr{phi:skip}(\|y\|),
\end{align*}
where
\begin{align*}
	\mathcal{E}(y):=\mleft\{
		T(0,y)<C_0\|y\| \text{ and }
		T_w^0(\underline{w},\overline{w})<2C_0\|y\|^{1/(2d+6)} \text{ for all } w \in \tilde{\gamma}_\mathcal{O}(y) \mright\}.
\end{align*}
\end{lem}

\begin{lem}\label{lem:longwander}
There exist a constant $A>0$ (which depends only on $d$ and $r_0$) and $\Cl[K]{phi:long} \in \mathcal{K}(d,r_0)$ such that if $\|y\|$ is large enough, then
\begin{align*}
	\sup_{r_0 \leq r \leq 1} \bP_r\bigl( \mathcal{E}(y) \cap \mathcal{E}'(A,y)^c \big| \omega(0)=1 \bigr) \leq \Cr{phi:long}(\|y\|),
\end{align*}
where
\begin{align*}
	\mathcal{E}'(A,y):=
	\biggl\{ \sum_{z \in \tilde{\gamma}_\mathcal{O}(y)} T_z^0(\underline{z},\overline{z})
	\1{\{ T_z^0(\underline{z},\overline{z})>A(\|\underline{z}-z\|+\|z-\overline{z}\|) \}} \leq \|y\| \biggr\}.
\end{align*}
\end{lem}

Before proving the lemmata above, let us complete the proof of Proposition~\ref{prop:influ_upper}.

\begin{proof}[\bf Proof of Proposition~\ref{prop:influ_upper}]
Assume that $\|y\|$ is sufficiently large (as required by Lemma~\ref{lem:longwander}) and let $N \geq C_0\|y\|$.
The following facts are immediate consequences of the definition of $\omega_z^s$ and the triangle inequality \eqref{eq:triangle}:
\begin{itemize}
	\item Both $T_z^0(0,y)$ and $T_z^1(0,y)$ are independent of $\omega(z)$.
	\item $T_z^0(0,y)=T_z^1(0,y)$ holds if there exists $\gamma \in \mathcal{O}(y,\omega_z^1,\S)$ such that $z \not\in \gamma$.
	\item We have $T_z^0(0,y)-T_z^1(0,y) \leq T_z^0(\underline{z},\overline{z})$ if $z \in \tilde{\gamma}_\mathcal{O}(y)$.
\end{itemize}
These facts show that for all $r \in [r_0,1]$,
\begin{align*}
	&\sum_{z \in B(0,N) \setminus \{0,y\}}\bE_r\mleft[ T_z^0(0,y)-T_z^1(0,y) \middle| \omega(0)=1 \mright]\\
	&\leq r^{-1} \times \sum_{z \in B(0,N) \setminus \{0,y\}}
		\bE_r\mleft[ T_z^0(\underline{z},\overline{z}) \1{\{ z \in \tilde{\gamma}_\mathcal{O}(y) \}} \middle| \omega(0)=1 \mright]\\
	&\leq r_0^{-1} \times \bE_r\biggl[ \sum_{z \in \tilde{\gamma}_\mathcal{O}(y)} T_z^0(\underline{z},\overline{z}) \bigg| \omega(0)=1 \biggr].
\end{align*}
From Lemma~\ref{lem:skip_negl}, the last expectation is not greater than
\begin{align*}
	\Cr{phi:skip}(\|y\|)
	+\bE_r\biggl[ \1{\mathcal{E}(y)} \sum_{z \in \tilde{\gamma}_\mathcal{O}(y)} T_z^0(\underline{z},\overline{z}) \bigg| \omega(0)=1 \biggr].
\end{align*}
Let us divide the second term in the above expression into two parts:
\begin{align*}
	\bE_r\biggl[ \1{\mathcal{E}(y)} \sum_{z \in \tilde{\gamma}_\mathcal{O}(y)} T_z^0(\underline{z},\overline{z}) \bigg| \omega(0)=1 \biggr]
	=I_1+I_2,
\end{align*}
where
\begin{align*}
	&I_1:=\bE_r\biggl[ \1{\mathcal{E}(y)} \sum_{z \in \tilde{\gamma}_\mathcal{O}(y)} T_z^0(\underline{z},\overline{z})
		\1{\{ T_z^0(\underline{z},\overline{z}) \leq A(\|\underline{z}-z\|+\|z-\overline{z}\|) \}}
		\bigg| \omega(0)=1 \biggr],\\
	&I_2:=\bE_r\biggl[ \1{\mathcal{E}(y)} \sum_{z \in \tilde{\gamma}_\mathcal{O}(y)} T_z^0(\underline{z},\overline{z})
		\1{\{ T_z^0(\underline{z},\overline{z})>A(\|\underline{z}-z\|+\|z-\overline{z}\|) \}}
		\bigg| \omega(0)=1 \biggr]
\end{align*}
(see Lemma~\ref{lem:longwander} for the choice of the constant $A$).
It is clear from the definitions of $\tilde{\gamma}_\mathcal{O}(y)$ that both $\sum_{z \in \tilde{\gamma}_\mathcal{O}(y)} \|\underline{z}-z\|$ and $\sum_{z \in \tilde{\gamma}_\mathcal{O}(y)} \|z-\overline{z}\|$ do not exceed $T(0,y)$.
This implies that
\begin{align*}
	I_1
	&\leq \bE_r\biggl[ \1{\mathcal{E}(y)}
		\sum_{z \in \tilde{\gamma}_\mathcal{O}(y)} A(\|\underline{z}-z\|+\|z-\overline{z}\|)
		\bigg| \omega(0)=1 \biggr]\\
	&\leq 2A \times \bE_r \mleft[ T(0,y)\1{\mathcal{E}(y)} \middle| \omega(0)=1 \mright]
		\leq 2AC_0\|y\|.
\end{align*}
On the other hand, the Cauchy--Schwarz inequality and Lemmata~\ref{lem:skip} and \ref{lem:longwander} imply that
\begin{align*}
	I_2
	&\leq \|y\|+\sum_{z \in B(0,C_0\|y\|)} \bE_r\mleft[ \1{\mathcal{E}(y) \cap \mathcal{E}'(A,y)^c}
		\times T_z^0(\underline{z},\overline{z}) \1{\{ z \in \tilde{\gamma}_\mathcal{O}(y) \}} \middle| \omega(0)=1 \mright]\\
	&\leq \|y\|+\Cr{skip}^{1/2}(2C_0\|y\|+1)^d\Cr{phi:long}(\|y\|)^{1/2}.
\end{align*}
With these observations, for any $r \in [r_0,1]$,
\begin{align*}
	&\sum_{z \in B(0,N) \setminus \{0,y\}}\bE_r\mleft[ T_z^0(0,y)-T_z^1(0,y) \middle| \omega(0)=1 \mright]\\
	&\leq r_0^{-1}\mleft\{ (2AC_0+1)\|y\|+\Cr{phi:skip}(\|y\|)+\Cr{skip}^{1/2}(2C_0\|y\|+1)^d \Cr{phi:long}(\|y\|)^{1/2} \mright\}.
\end{align*}
Therefore, the proposition immediately follows because $\Cr{phi:skip},\Cr{phi:long} \in \mathcal{K}(d,r_0)$ and the constants $C_0$, $\Cr{skip}$ and $A$ depend only on $d$ and $r_0$.
\end{proof}

The remainder of this subsection is devoted to the proofs of Lemmata~\ref{lem:skip}, \ref{lem:skip_negl} and \ref{lem:longwander}.

\begin{proof}[\bf Proof of Lemma~\ref{lem:skip}]
Let $y,z \in \Z^d \setminus \{0\}$ and $r \in [r_0,1]$.
The Cauchy--Schwarz inequality gives
\begin{align}\label{eq:skip}
\begin{split}
	&\bE_r\mleft[ T_z^0(\underline{z},\overline{z})^2 \1{\{ z \in \tilde{\gamma}_\mathcal{O}(y) \}} \middle| \omega(0)=1 \mright]\\
	&\leq \sum_{\substack{w_1,w_2 \in \Z^d \setminus \{z\}\\ w_1 \not=w_2}}
		\bE_r\mleft[ T_z^0(w_1,w_2)^4 \1{\{ \omega(w_1)=1 \}} \middle| \omega(0)=1 \mright]^{1/2}\\
	&\qquad\qquad\quad \times
		\bP_r\mleft( z \in \tilde{\gamma}_\mathcal{O}(y),\,\underline{z}=w_1,\,\overline{z}=w_2 \middle| \omega(0)=1 \mright)^{1/2}.
\end{split}
\end{align}
Let us first estimate the expectations in the right side of \eqref{eq:skip}.
From Lemma~\ref{lem:AP}, there exists a constant $c$ (which depends only on $d$ and $r_0$) such that for all $w_1,w_2 \in \Z^d \setminus \{z\}$ with $w_1 \not=w_2$,
\begin{align}\label{eq:skip_E}
\begin{split}
	&\bE_r\mleft[ T_z^0(w_1,w_2)^4 \1{\{ \omega(w_1)=1 \}} \middle| \omega(0)=1 \mright]\\
	&\leq \bE_r\mleft[ T_{z-w_1}^0(0,w_2-w_1)^4 \middle| \omega(0)=1 \mright]
		\leq c\|w_1-w_2\|^4.
\end{split}
\end{align}
We next estimate the probabilities in the right side of \eqref{eq:skip}.
Note that $T(u,v)=\tau(u,v) \geq \|u-v\|$ holds if $u$ and $v$ are consecutive points in $\gamma_\mathcal{O}(y)$.
This together with Lemma~\ref{lem:next} yields that for any $w_1 \in \Z^d \setminus \{z\}$,
\begin{align*}
	&\bP_r(z \in \tilde{\gamma}_\mathcal{O}(y),\,\underline{z}=w_1|\omega(0)=1)\\
	&\leq \bP_r\bigl( T(w_1,z)=\tau(w_1,z) \geq \|w_1-z\| \big| \omega(0)=1 \bigr)\\
	&\leq \bP_r\bigl( T(0,z-w_1)=\tau(0,z-w_1) \geq \|w_1-z\| \big| \omega(0)=1 \bigr)
		\leq \Cr{phi:next}(\|w_1-z\|).
\end{align*}
Similarly, for any $w_2 \in \Z^d \setminus \{z\}$,
\begin{align*}
	\bP_r\mleft( z \in \tilde{\gamma}_\mathcal{O}(y),\,\overline{z}=w_2 \middle| \omega(0)=1 \mright)
	\leq \Cr{phi:next}(\|w_2-z\|).
\end{align*}
Hence, the Cauchy--Schwarz inequality tells us that for any $w_1,w_2 \in \Z^d \setminus \{z\}$ with $w_1 \not= w_2$,
\begin{align}\label{eq:skip_P}
\begin{split}
	&\bP_r\mleft( z \in \tilde{\gamma}_\mathcal{O}(y),\,\underline{z}=w_1,\,\overline{z}=w_2 \middle| \omega(0)=1 \mright)\\
	&\leq \bP_r\mleft( z \in \tilde{\gamma}_\mathcal{O}(y),\,\underline{z}=w_1 \middle| \omega(0)=1 \mright)^{1/2}
		\,\bP_r\mleft( z \in \tilde{\gamma}_\mathcal{O}(y),\,\overline{z}=w_2 \middle| \omega(0)=1 \mright)^{1/2}\\
	&\leq \Cr{phi:next}(\|w_1-z\|)^{1/2}\,\Cr{phi:next}(\|w_2-z\|)^{1/2}.
\end{split}
\end{align}

Consequently, by \eqref{eq:skip}, \eqref{eq:skip_E} and \eqref{eq:skip_P},
\begin{align*}
	&\bE_r\mleft[ T_z^0(\underline{z},\overline{z})^2 \1{\{ z \in \tilde{\gamma}_\mathcal{O}(y) \}} \middle| \omega(0)=1 \mright]\\
	&\leq c^{1/2}\sum_{\substack{w_1,w_2 \in \Z^d \setminus \{z\}\\ w_1 \not=w_2}}
		\|w_1-w_2\|^2\,\Cr{phi:next}(\|w_1-z\|)^{1/4}\,\Cr{phi:next}(\|w_2-z\|)^{1/4}.
\end{align*}
A simple calculation shows that the right side is bounded from above by
\begin{align*}
	&2c^{1/2}\sum_{\substack{w_1,w_2 \in \Z^d \setminus \{z\}\\ w_1 \not=w_2}}
		\bigl( \|w_1-z\|^2+\|w_2-z\|^2 \bigr)\,\Cr{phi:next}(\|w_1-z\|)^{1/4}\,\Cr{phi:next}(\|w_2-z\|)^{1/4}\\
	&\leq 4c^{1/2}\sum_{w_1 \in \Z^d} \|w_1-z\|^2\,\Cr{phi:next}(\|w_1-z\|)^{1/4}
		\times \sum_{w_2 \in \Z^d}\Cr{phi:next}(\|w_2-z\|)^{1/4}.
\end{align*}
Since $\Cr{phi:next} \in \mathcal{K}(d,r_0)$ and the constant $c$ depends only on $d$ and $r_0$, the right side is finite and depends only on $d$ and $r_0$.
Hence, the lemma follows. 
\end{proof}

\begin{proof}[\bf Proof of Lemma~\ref{lem:skip_negl}]
To shorten notation, set $\delta:=1/(2d+6)$.
For the proof, it suffices to show that there exist $\psi_1,\psi_2 \in \mathcal{K}(d,r_0)$ such that for all $y \in \Z^d \setminus \{0\}$ and $r \in [r_0,1]$,
\begin{align}
	&\label{eq:skip_negl1}
		\bE_r\biggl[ \1{\{ T(0,y) \geq C_0\|y\| \}} \sum_{z \in \tilde{\gamma}_\mathcal{O}(y)}
		T_z^0(\underline{z},\overline{z}) \bigg| \omega(0)=1 \biggr]
		\leq \psi_1(\|y\|),\\[0.5em]
	&\label{eq:skip_negl2}
		\bP_r\mleft(
		\begin{minipage}{14.5em}
			$T(0,y)<C_0\|y\|$ and
			$\exists w \in \tilde{\gamma}_\mathcal{O}(y)$\\
			such that $T_w^0(\underline{w},\overline{w}) \geq 2C_0\|y\|^\delta$
		\end{minipage}
		\middle| \omega(0)=1 \mright)
		\leq \psi_2(\|y\|).
\end{align}
Indeed, combining the above inequalities with Lemma~\ref{lem:skip}, we have for all $y \in \Z^d \setminus \{0\}$ and $r \in [r_0,1]$,
\begin{align*}
	&\bE_r\biggl[ \1{\mathcal{E}(y)^c} \sum_{z \in \tilde{\gamma}_\mathcal{O}(y)}
		T_z^0(\underline{z},\overline{z}) \bigg| \omega(0)=1 \biggr]\\
	&\leq \psi_1(\|y\|)
		+\sum_{z \in B(0,C_0\|y\|)}
		\bE_r\mleft[ T_z^0(\underline{z},\overline{z})^2\1{\{ z \in \tilde{\gamma}_\mathcal{O}(y) \}} \middle| \omega(0)=1 \mright]^{1/2} \,\psi_2(\|y\|)^{1/2}\\
	&\leq \psi_1(\|y\|)+\Cr{skip}^{1/2} (2C_0\|y\|+1)^d \psi_2(\|y\|)^{1/2},
\end{align*}
and the lemma follows since $\psi_1,\psi_2 \in \mathcal{K}(d,r_0)$ and the constants $C_0$ and $\Cr{skip}$ depend only on $d$ and $r_0$.

Let us first prove \eqref{eq:skip_negl1}.
The left side of \eqref{eq:skip_negl1} is rewritten as follows:
\begin{align*}
	\bE_r\biggl[ \1{\{ T(0,y) \geq C_0\|y\| \}} \sum_{z \in \tilde{\gamma}_\mathcal{O}(y)}
	T_z^0(\underline{z},\overline{z}) \bigg| \omega(0)=1 \biggr]
	= J_1+J_2,
\end{align*}
where
\begin{align*}
	&J_1:=\sum_{z \in B(0,C_0\|y\|)^c}
		\bE_r\biggl[ T_z^0(\underline{z},\overline{z}) \1{\{ z \in \tilde{\gamma}_\mathcal{O}(y),\,T(0,y) \geq \|z\| \}}
		\bigg| \omega(0)=1 \biggr],\\
	&J_2:=\sum_{z \in B(0,C_0\|y\|)}
		\bE_r\biggl[ T_z^0(\underline{z},\overline{z})\1{\{ z \in \tilde{\gamma}_\mathcal{O}(y),\,T(0,y) \geq C_0\|y\| \}}
		\bigg| \omega(0)=1 \biggr].
\end{align*}
The Cauchy--Schwarz inequality and Lemmata~\ref{lem:AP} and \ref{lem:skip} imply that
\begin{align*}
	J_1
	&\leq \sum_{z \in B(0,C_0\|y\|)^c} \bE_r\mleft[ T_z^0(\underline{z},\overline{z})^2
			\1{\{ z \in \tilde{\gamma}_\mathcal{O}(y) \}} \middle| \omega(0)=1 \mright]^{1/2}\\
	&\qquad\qquad\qquad\qquad \times \bP_r(T(0,y) \geq \|z\||\omega(0)=1)^{1/2}\\
	&\leq \Cr{skip}^{1/2} \biggl( \sum_{z \in \Z^d}\phi_0(\|z\|)^{1/4} \biggr) \phi_0(C_0\|y\|)^{1/4}
\end{align*}
and
\begin{align*}
	J_2
	&\leq \sum_{z \in B(0,C_0\|y\|)}
		\bE_r\mleft[ T_z^0(\underline{z},\overline{z})^2 \1{\{ z \in \tilde{\gamma}_\mathcal{O}(y) \}}
		\middle| \omega(0)=1 \mright]^{1/2}\\
	&\qquad\qquad\qquad\qquad \times \bP_r(T(0,y) \geq C_0\|y\||\omega(0)=1)^{1/2}\\		
	&\leq \Cr{skip}^{1/2}(2C_0\|y\|+1)^d \phi_0(C_0\|y\|)^{1/2}.
\end{align*}
Note that $\phi_0 \in \mathcal{K}(d,r_0)$ and the constants $C_0$ and $\Cr{skip}$ depend only on $d$ and $r_0$.
Hence, we can find $\psi_1 \in \mathcal{K}(d,r_0)$ such that $J_1+J_2 \leq \psi_1(\|y\|)$, and \eqref{eq:skip_negl1} follows.

Before moving on to the proof of \eqref{eq:skip_negl2}, we observe that there exists $\psi \in \mathcal{K}(d,r_0)$ such that for any $y,w \in \Z^d \setminus \{0\}$ and $r \in [r_0,1]$,
\begin{align}\label{eq:skip_negl3}
\begin{split}
	&\bP_r\mleft( w \in \tilde{\gamma}_\mathcal{O}(y)
		\text{ and } \|\underline{w}-w\| \vee \|w-\overline{w}\|>\|y\|^\delta
		\middle| \omega(0)=1 \mright)\\
	&\leq \psi(\|y\|).
\end{split}
\end{align}
This is a direct consequence of the union bound and the same argument as in the proof of Lemma~\ref{lem:skip} (see below \eqref{eq:skip_E}).
Indeed, for any $y,w \in \Z^d \setminus \{0\}$ and $r \in [r_0,1]$, the left side of \eqref{eq:skip_negl3} is bounded from above by
\begin{align*}
	&\sum_{x \in B(w,\|y\|^\delta)^c}
		\bP_r(w \in \tilde{\gamma}_\mathcal{O}(y),\,\underline{w}=x|\omega(0)=1)\\
	&+\sum_{x \in B(w,\|y\|^\delta)^c}
			\bP_r(w \in \tilde{\gamma}_\mathcal{O}(y),\,\overline{w}=x|\omega(0)=1)\\
	&\leq 2\biggl( \sum_{x \in \Z^d} \Cr{phi:next}(\|x\|)^{1/2} \biggr) \Cr{phi:next}(\|y\|^\delta)^{1/2}.
\end{align*}
Hence, due to $\Cr{phi:next} \in \mathcal{K}(d,r_0)$, we can find $\psi \in \mathcal{K}(d,r_0)$ satisfying \eqref{eq:skip_negl3}.

Let us finally prove \eqref{eq:skip_negl2}.
By the union bound and \eqref{eq:skip_negl3}, the left side of \eqref{eq:skip_negl2} is smaller than or equal to
\begin{align*}
	&\sum_{w \in B(0,C_0\|y\|) \setminus \{0\}} \bP_r\mleft( w \in \tilde{\gamma}_\mathcal{O}(y),\,
		T_w^0(\underline{w},\overline{w}) \geq 2C_0\|y\|^\delta \middle| \omega(0)=1 \mright)\\
	&\leq (2C_0\|y\|+1)^d\psi(\|y\|)\\
	&\quad +\sum_{w \in B(0,C_0\|y\|)}
		\sum_{x_1,x_2 \in B(w,\|y\|^\delta) \setminus \{w\}}
		\bP_r\mleft( T_{w-x_1}^0(0,x_2-x_1) \geq 2C_0\|y\|^\delta \middle| \omega(0)=1 \mright).
\end{align*}
Note that for any $w \in \Z^d$ and $x_1,x_2 \in B(w,\|y\|^\delta)$,
\begin{align*}
	C_0\|x_2-x_1\| \leq C_0(\|x_2-w\|+\|w-x_1\|) \leq 2C_0\|y\|^\delta.
\end{align*}
Hence, Lemma~\ref{lem:AP} implies that for all $w \in \Z^d$ and $x_1,x_2 \in B(w,\|y\|^\delta) \setminus \{w\}$ with $x_1 \not= x_2$,
\begin{align*}
	\bP_r\mleft( T_{w-x_1}^0(0,x_2-x_1) \geq 2C_0\|y\|^\delta \middle| \omega(0)=1 \mright)
	\leq \phi_0(2C_0\|y\|^\delta).
\end{align*}
With these observations, the left side of \eqref{eq:skip_negl2} does not exceed
\begin{align*}
	(2C_0\|y\|+1)^d\psi(\|y\|)+(2C_0\|y\|+1)^{3d}\phi_0(2C_0\|y\|^\delta).
\end{align*}
Since $\phi_0,\psi \in \mathcal{K}(d,r_0)$ and $C_0$ depends only on $d$ and $r_0$, it is easy to see that \eqref{eq:skip_negl2} holds for some $\psi_2 \in \mathcal{K}(d,r_0)$.
\end{proof}

\begin{proof}[\bf Proof of Lemma~\ref{lem:longwander}]
Continuing from the proof of Lemma~\ref{lem:skip_negl}, we use the notation $\delta=1/(2d+6)$.
Take $A \geq 16C_0^3\Cr{CN}$ large enough to satisfy that for all $t>A$,
\begin{align*}
	\phi_0(t) \leq C_0^{-d}\Cr{CN}^{-d}(4t)^{-d(d+6)}.
\end{align*}
Moreover, fix $y \in \Z^d$ with $2C_0\|y\|^\delta>A+1$ and $r \in [r_0,1]$.
Now, assume that the event $\mathcal{E}(y) \cap \mathcal{E}'(A,y)^c$ occurs.
Then, $\gamma_\mathcal{O}(y) \in \Pi_{C_0\|y\|}$ (see above Proposition~\ref{prop:CN} for the notation $\Pi_{C_0\|y\|}$) and
\begin{align}\label{eq:bigjump}
\begin{split}
	\|y\|
	&< \sum_{A<\ell<2C_0\|y\|^\delta} \ell \sum_{z \in \tilde{\gamma}_\mathcal{O}(y)}
		\1{\{ T_z^0(\underline{z},\overline{z})=\ell,\,\|\underline{z}-z\|+\|z-\overline{z}\|<\ell/A \}}.
\end{split}
\end{align}
To estimate the right side of \eqref{eq:bigjump}, for each $\ell \in \N$, let us consider the following family $(X_\ell(z))_{z \in \Z^d}$ of $4\ell$-dependent random variables and parameter $p_\ell$:
\begin{align*}
	X_\ell(z):=\1{\{ \exists w_1,w_2 \in B(z,\ell/A) \text{ such that } T_z^0(w_1,w_2)=\ell \}},
	\qquad z \in \Z^d,
\end{align*}
and
\begin{align*}
	p_\ell:=\sup_{z \in \Z^d}\bP_r(X_\ell(z)=1)=\bP_r(X_\ell(0)=1).
\end{align*}
Since $\gamma_\mathcal{O}(y) \in \Pi_{C_0\|y\|}$ and
\begin{align*}
	\1{\{ T_z^0(\underline{z},\overline{z})=\ell,\,\|\underline{z}-z\|+\|z-\overline{z}\|<\ell/A \}} \leq X_\ell(z),
	\qquad z \in \tilde{\gamma}_\mathcal{O}(y),
\end{align*}
the right side of \eqref{eq:bigjump} is smaller than or equal to
\begin{align*}
	\sum_{A<\ell<2C_0\|y\|^\delta} \ell \sum_{z \in \tilde{\gamma}_\mathcal{O}(y)} X_\ell(z)
	\leq \sum_{A<\ell<2C_0\|y\|^\delta} \ell \max_{\pi \in \Pi_{C_0\|y\|}}\sum_{z \in \pi} X_\ell(z).
\end{align*}
With these observations, on $\mathcal{E}(y) \cap \mathcal{E}'(A,y)^c$,
\begin{align*}
	\|y\|<\sum_{A<\ell<2C_0\|y\|^\delta} \ell \max_{\pi \in \Pi_{C_0\|y\|}}\sum_{z \in \pi} X_\ell(z).
\end{align*}
Hence, the fact that $\sum_{A<\ell<2C_0\|y\|^\delta} \ell^{-2} \leq 1$ and the union bound imply that
\begin{align}\label{eq:UBE}
\begin{split}
	&\bP_r\bigl( \mathcal{E}(y) \cap \mathcal{E}'(A,y)^c \big| \omega(0)=1 \bigr)\\
	&\leq r_0^{-1}\sum_{A<\ell<2C_0\|y\|^\delta} \bP_r\biggl(
		\max_{\pi \in \Pi_{C_0\|y\|}} \sum_{z \in \pi} X_\ell(z)>\ell^{-3}\|y\| \biggr).
\end{split}
\end{align}
Note that the choice of $A$ and Lemma~\ref{lem:AP} guarantee that for all $\ell \in (A,2C_0\|y\|^\delta)$,
\begin{align*}
	p_\ell
	&\leq \sum_{w_1,w_2 \in B(0,\ell/A)}\bP_r\bigl( T_{-w_1}^0(0,w_2-w_1)=\ell \big| \omega(0)=1 \bigr)\\
	&\leq (2\ell+1)^{2d}\phi_0(\ell)
	\leq (12\ell+1)^{-d}
\end{align*}
and
\begin{align*}
	\ell^{-3}\|y\| \geq \Cr{CN}(4\ell)^{d+1}\max\bigl\{ 1,C_0\|y\|p_\ell^{1/d} \bigr\}.
\end{align*}
It follows from Proposition~\ref{prop:CN} that there exists a constant $c$ (which depends only on $d$ and $r_0$) such that for all $\ell \in (A,2C_0\|y\|^\delta)$, 
\begin{align*}
	\bP_r\biggl( \max_{\pi \in \Pi_{C_0\|y\|}} \sum_{z \in \pi} X_\ell(z)>\ell^{-3}\|y\| \biggr)
	&\leq 2^d \exp\mleft\{ -\frac{\ell^{-3}\|y\|}{(64\ell)^d} \mright\}\\
	&\leq 2^d \exp\bigl\{ -c\|y\|^{1/2} \bigr\}.
\end{align*}
This together with \eqref{eq:UBE} proves that
\begin{align*}
	\bP_r\bigl( \mathcal{E}(y) \cap \mathcal{E}'(A,y)^c \big| \omega(0)=1 \bigr)
	\leq 2^{d+1}r_0^{-1}C_0\|y\| \exp\bigl\{ -c\|y\|^{1/2} \bigr\}.
\end{align*}
It is clear that there exists $\Cr{phi:long} \in \mathcal{K}(d,r_0)$ such that for all $t \geq 0$,
\begin{align*}
	2^{d+1}r_0^{-1}C_0t \exp\bigl\{ -ct^{1/2} \bigr\} \leq \Cr{phi:long}(t),
\end{align*}
and the proof is complete.
\end{proof}

\subsection{Proof of Proposition~\ref{prop:influ_lower}}\label{subsect:influ_lower}
Let us first introduce events describing delays in the propagation of active frogs.
For any $x \in \Z^d$ and subset $\Gamma$ of $\R^d$, write $T(x,\Gamma)$ for the first passage time from $x$ to $\Gamma$, i.e.,
\begin{align}\label{eq:point-set}
	T(x,\Gamma):=\inf_{y \in \Gamma} T(x,y).
\end{align}
Moreover, for any $L \in \N$ and $z \in \Z^d$, let $\mathcal{D}_L(z):=\{ w \in \Z^d:\|w-z\|=L \}$ be the $\ell^1$-sphere of center $z$ and radius $L$.
Then, for each $L \in \N$ and $y,z \in \Z^d \setminus \{0\}$, we consider the following event $\mathcal{E}(L,y,z)$:
\begin{align*}
	\mathcal{E}(L,y,z)
	:= \mleft\{ z \in \gamma_\mathcal{O}(y),\,T(z,\mathcal{D}_L(z))>L,\,\|\overline{z}^{(L)}-z\| \leq 2L \mright\},
\end{align*}
where $\overline{z}^{(L)}$ stands for the first point $w$ after $z$ along $\gamma_\mathcal{O}(y)$ satisfying that $\|w-z\|>L$ if such a point $w$ exists; otherwise set $\overline{z}^{(L)}:=z+3L\xi_1$.
The event $\mathcal{E}(L,y,z)$ can be regarded as a ``delaying'' event.
Actually, on $\mathcal{E}(L,y,z)$, the frog sitting on $z$ contributes the first passage time from $0$ to $y$, but it takes at least time $L+1$ for active frogs propagated from $z$ to reach $\mathcal{D}_L(z)$.
Since the minimum time from $z$ to $\mathcal{D}_L(z)$ is $L$, the occurrence of $\mathcal{E}(L,y,z)$ delays the arrival of active frogs at $y$.
Although the condition $\| \bar{z}^{(L)}-z\| \leq 2L$ appearing in $\mathcal{E}(L,y,z)$ is not directly related to the delay, it enables us to connect $z$ to $\bar{z}^{(L)}$ with uniform probability by using the frog sitting on $z$.

The key to proving Proposition~\ref{prop:influ_lower} is the following proposition, which gives a lower bound for sums of probabilities of delaying events.

\begin{prop}\label{prop:delay}
There exists $L \in \N$ (which depends only on $d$ and $r_0$) such that if $\|y\|$ is large enough and $N \geq C_0\|y\|$, then for all $r \in [r_0,1]$,
\begin{align*}
	\sum_{z \in B(0,N) \setminus \{0,y\}} \bP_r(\mathcal{E}(L,y,z)|\omega(0)=1)
	\geq \frac{\|y\|}{4L}.
\end{align*}
\end{prop}

The proof of the above proposition is postponed until the next section since we need some more work.
For now, let us move on to the proof of Proposition~~\ref{prop:influ_lower}.

\begin{proof}[\bf Proof of Proposition~~\ref{prop:influ_lower}]
Let $\|y\|$ be sufficiently large and $N \geq C_0\|y\|$.
It is clear that if every $\gamma \in \mathcal{O}(y)$ contains a site $z$, then $T_z^0(0,y)-T_z^1(0,y) \geq 1$ holds.
This implies that for any $r \in [r_0,1]$,
\begin{align*}
	&\sum_{z \in B(0,N) \setminus \{0,y\}} \bE_r\mleft[ T_z^0(0,y)-T_z^1(0,y) \middle| \omega(0)=1 \mright]\\
	&\geq \sum_{z \in B(0,N) \setminus \{0,y\}}
		\bP_r\mleft( z \in \gamma \text{ for all } \gamma \in \mathcal{O}(y) \middle| \omega(0)=1 \mright).
\end{align*}
Hence, Proposition~~\ref{prop:influ_lower} immediately follows once we prove that for any $r \in [r_0,1]$ and $z \in \Z^d \setminus \{0\}$,
\begin{align}\label{eq:resampling}
	\bP_r(z \in \gamma \text{ for all } \gamma \in \mathcal{O}(y)|\omega(0)=1)
	\geq (2d)^{-2L}\,\bP_r(\mathcal{E}(L,y,z)|\omega(0)=1).
\end{align}
Indeed, this combined with Proposition~\ref{prop:delay} shows that for all $r \in [r_0,1]$,
\begin{align*}
	&\sum_{z \in B(0,N) \setminus \{0,y\}} \bE_r\mleft[ T_z^0(0,y)-T_z^1(0,y) \middle| \omega(0)=1 \mright]\\
	&\geq (2d)^{-2L}\sum_{z \in B(0,N) \setminus \{0,y\}} \bP_r(\mathcal{E}(L,y,z)|\omega(0)=1)
	\geq \frac{(2d)^{-2L}}{4L} \|y\|,
\end{align*}
which is the desired conclusion since $L$ depends only on $d$ and $r_0$.

It remains to prove \eqref{eq:resampling}.
Fix $r \in [r_0,1]$ and $z \in \Z^d \setminus \{0\}$.
Let $S_\cdot^*(z)$ be an independent copy of $S_\cdot(z)$, and we put the superscript $*$ on the notations $\tau(\cdot,\cdot)$, $T(\cdot,\cdot)$ and $\mathcal{O}(y)$ when $S_\cdot(z)$ is replaced by $S^*_\cdot(z)$ in their definitions.
Now, consider the event
\begin{align*}
	\bar{\mathcal{E}}(L,y,z):=\mathcal{E}(L,y,z) \cap \mleft\{ \tau^*(z,\overline{z}^{(L)})=\|z-\overline{z}^{(L)}\| \mright\}.
\end{align*}
Suppose that on $\bar{\mathcal{E}}(L,y,z)$, there exists a sequence in $\mathcal{O}^*(y)$ not containing $z$.
Clearly, $T(0,y) \leq T^*(0,y)$ holds.
Furthermore, letting $w_0$ be the first point of $\mathcal{D}_L(z)$ which is visited by the frogs in the segment between $z$ and $\bar{z}^{(L)}$ of $\gamma_\mathcal{O}(y)$, one has
\begin{align*}
	T(z,\overline{z}^{(L)})
	&\geq T(z,w_0)+\|w_0-\overline{z}^{(L)}\|\\
	&\geq L+1+\|w_0-\overline{z}^{(L)}\|
		= \|z-w_0\|+\|w_0-\overline{z}^{(L)}\|+1\\
	&\geq \|z-\overline{z}^{(L)}\|+1
		> \tau^*(z,\overline{z}^{(L)}).
\end{align*}
Denote by $(x_i)_{i=0}^j$ (resp.~$(x'_i)_{i=0}^{j'}$) the segment of $\gamma_\mathcal{O}(y)$ with endpoints $0$ and $z$ (resp.~$\bar{z}^{(L)}$ and $y$).
Then, due to the optimality of $\gamma_\mathcal{O}(y)$,
\begin{align*}
	T(0,y) \leq T^*(0,y)
	&\leq \sum_{i=0}^{j-1} \tau(x_i,x_{i+1})+\tau^*(z,\overline{z}^{(L)})+\sum_{i=0}^{j'-1}\tau(x'_i,x'_{i+1})\\
	&< T(0,z)+T(z,\overline{z}^{(L)})+T(\overline{z}^{(L)},y)=T(0,y),
\end{align*}
which is a contradiction.
This means that every $\gamma \in \mathcal{O}^*(y)$ contains $z$ on $\bar{\mathcal{E}}(L,y,z)$, and we have
\begin{align}\label{eq:P*}
	\bar{\bP}_r^z( z \in \gamma \text{ for all } \gamma \in \mathcal{O}^*(y)|\omega(0)=1)
	\geq \bar{\bP}_r^z\bigl( \bar{\mathcal{E}}(L,y,z) \big| \omega(0)=1 \bigr),
\end{align}
where $\bar{\bP}_r^z$ is the product of $\bP_r$ and the law of $S_\cdot^*(z)$.
Due to $z \in \Z^d \setminus \{0\}$ and the definition of $\bar{\bP}_r^z$,
\begin{align*}
	\bar{\bP}_r^z( z \in \gamma \text{ for all } \gamma \in \mathcal{O}^*(y)|\omega(0)=1)
	= \bP_r( z \in \gamma \text{ for all } \gamma \in \mathcal{O}(y)|\omega(0)=1)
\end{align*}
and
\begin{align*}
	&\bar{\bP}_r^z\bigl( \bar{\mathcal{E}}(L,y,z) \big| \omega(0)=1 \bigr)\\
	&\geq \bP_r( \mathcal{E}(L,y,z)|\omega(0)=1)
		\times \inf_{w \in B(0,2L)}P^0(\tau(0,w)=\|w\|)\\
	&\geq (2d)^{-2L}\,\bP_r( \mathcal{E}(L,y,z)|\omega(0)=1).
\end{align*}
Combining the above facts and \eqref{eq:P*}, we obtain
\begin{align*}
	\bP_r( z \in \gamma \text{ for all } \gamma \in \mathcal{O}(y)|\omega(0)=1)
	\geq (2d)^{-2L}\,\bP_r( \mathcal{E}(L,y,z)|\omega(0)=1),
\end{align*}
and \eqref{eq:resampling} follows.
\end{proof}

\section{Estimate for sums of probabilities of delaying events}\label{sect:delay}
This section is devoted to the proof of Proposition~\ref{prop:delay}.
Roughly speaking, our main task is to observe that the first passage time is strictly greater than the $\ell^1$-norm with high probability and there are enough $\ell^1$-balls of a sufficiently large radius that intersect $\gamma_\mathcal{O}(y)$ but contain no long-wandering frog.
Since the first assertion is of interest in itself, we discuss it in Section~\ref{subsect:Nakajima}, apart from the proof of Proposition~\ref{prop:delay}.
The second assertion and the proof of Proposition~\ref{prop:delay} are dealt with in Section~\ref{subsect:pf_delay}.

\subsection{Comparison between the first passage time and the $\boldsymbol{\ell^1}$-norm}\label{subsect:Nakajima}
This subsection treats the case $r=1$.
In other words, throughout this subsection, we assume that $\omega(x)=1$ for all $x \in \Z^d$.
Then, the product measure $\bP_1=\P_1 \times P$ can be reduced to the law $P$ of frogs because the initial configuration is no longer random.

Our objective here is to prove the next proposition, which provides an exponential decay for the probability that the first passage time coincides with the $\ell^1$-norm.

\begin{prop}\label{prop:Nakajima}
There exists a constant $\Cl{Nakajima}>0$ (which depends only on $d$) such that for all large $L \in \N$,
\begin{align}\label{eq:Nakajima}
	P\bigl( T(0,\mathcal{D}_L(0))=L \bigr) \leq e^{-\Cr{Nakajima}L},
\end{align}
where $T(0,\mathcal{D}_L(0))$ is the first passage time from $0$ to $\mathcal{D}_L(0)$ (see \eqref{eq:point-set} and the sentence following it for more details).
\end{prop}
\begin{proof}
Let us first prepare some notation.
Define for any $m \in \N_0$,
\begin{align*}
	\Delta_m:=\{ x \in (\N_0)^d:\|x\|=m \},\quad
	R_m:=\{ x \in (\N_0)^d:\|x\| \leq m \},
\end{align*}
which are the nonnegative orthants of the $\ell^1$-sphere and ball of center $0$ and radius $m$, respectively.
Furthermore, for any $x \in \Z^d$ and subset $\Gamma$ of $\Z^d$, $\tau(x,\Gamma)$ stands for the hitting time of the frog sitting on $x$ to $\Gamma$, i.e.,
\begin{align*}
	\tau(x,\Gamma):=\inf_{y \in \Gamma} \tau(x,y).
\end{align*}
Then, set for each $n \in \N$,
\begin{align*}
	q_n:=E\Biggl[ \sum_{y \in R_{n-1}}\1{\{ T(0,y)+\tau(y,\Delta_n)=n \}} \Biggr].
\end{align*}
This is the expectation of the number of frogs which finally approach $\Delta_n$ in each optimal selection and order for $T(0,\Delta_n)$, conditioned on the event that $T(0,\Delta_n)=n$.

The key to proving the proposition is the following lemmata, which provide some upper bounds for $q_n$'s. 

\begin{lem}\label{lem:q3}
We have $2q_3<1$.
\end{lem}

\begin{lem}\label{lem:q_induction}
Let $2 \leq A \in \N$.
Then, $q_{An} \leq 2^{-1}(2q_A)^n$ holds for all $n \in \N$.
\end{lem}

The proofs of the above lemmata are postponed for now, and we continue proving the proposition.
The symmetry of frogs and Lemma~\ref{lem:q_induction} imply that for any $L \geq 3$,
\begin{align*}
	P\bigl( T(0,\mathcal{D}_L(0))=L \bigr)
	&\leq 2^d P\bigl( T(0,\Delta_{3\lfloor L/3 \rfloor})=3\lfloor L/3 \rfloor \bigr)\\
	&\leq 2^dq_{3\lfloor L/3 \rfloor}
		\leq 2^{d-1}(2q_3)^{\lfloor L/3 \rfloor}.
\end{align*}
From Lemma~\ref{lem:q3}, the rightmost side exponentially decays in $L$, and the proposition follows.
\end{proof}

We close this subsection with the proofs of Lemmata~\ref{lem:q3} and \ref{lem:q_induction}.

\begin{proof}[\bf Proof of Lemma~\ref{lem:q3}]
By definition,
\begin{align*}
	q_3
	= P(\tau(0,\Delta_3)=3)&+\sum_{y \in \Delta_1} P\bigl( T(0,y)+\tau(y,\Delta_3)=3 \bigr)\\
	&+\sum_{y \in \Delta_2} P\bigl( T(0,y)+\tau(y,\Delta_3)=3 \bigr).
\end{align*}
A straightforward calculation gives $P(\tau(0,\Delta_3)=3)=1/8$.
Furthermore, we use the independence of frogs to obtain
\begin{align*}
	\sum_{y \in \Delta_1} P\bigl( T(0,y)+\tau(y,\Delta_3)=3 \bigr)
	&= \sum_{y \in \Delta_1} P(\tau(0,y)=1)\,P(\tau(y,\Delta_3)=2)\\
	&= \sum_{y \in \Delta_1} \frac{1}{2d} \times \frac{1}{4}=\frac{1}{8}
\end{align*}
and
\begin{align*}
	\sum_{y \in \Delta_2} P\bigl( T(0,y)+\tau(y,\Delta_3)=3 \bigr)
	&= \sum_{y \in \Delta_2} P(T(0,y)=2)\,P(\tau(y,\Delta_3)=1)\\
	&= \half \sum_{y \in \Delta_2} P(T(0,y)=2).
\end{align*}
With these observations,
\begin{align*}
	2q_3=\half+\sum_{y \in \Delta_2} P(T(0,y)=2).
\end{align*}
Hence, it suffices for the proof to show
\begin{align}\label{eq:half}
	\sum_{y \in \Delta_2} P(T(0,y)=2)<\half.
\end{align}

To do this, we rewrite the left side of \eqref{eq:half} as
\begin{align}\label{eq:half_key}
	\sum_{1 \leq i<j \leq d} P(T(0,\xi_i+\xi_j)=2)+\sum_{i=1}^dP(T(0,2\xi_i)=2).
\end{align}
The union bound gives a rough upper bound for the first term of \eqref{eq:half_key}:
\begin{align*}
	&\sum_{1 \leq i<j \leq d} P(T(0,\xi_i+\xi_j)=2)\\
	&\leq \sum_{1 \leq i<j \leq d} \biggl\{ P(\tau(0,\xi_i+\xi_j)=2)+\sum_{x \in \Delta_1} P\bigl( \tau(0,x)+\tau(x,\xi_i+\xi_j)=2 \bigr) \biggr\}\\
	&= \begin{pmatrix} d\\ 2\end{pmatrix} \frac{4}{(2d)^2}=\frac{2d(d-1)}{(2d)^2}.
\end{align*}
On the other hand, we calculate the second term of \eqref{eq:half_key} carefully.
It is clear that for any $i \in [1,d]$, the event $\{ T(0,2\xi_i)=2 \}$ can be written as the disjoint union of the following events $\mathcal{S}_i$, $\mathcal{T}_i$ and $\mathcal{U}_i$:
\begin{align*}
	&\mathcal{S}_i:= \{ S_1(0)=\xi_i,\,S_2(0)=2\xi_i,\,S_1(\xi_i) \not= 2\xi_i \},\\
	&\mathcal{T}_i:= \{ S_1(0)=\xi_i,\,S_2(0) \not= 2\xi_i,\,S_1(\xi_i)=2\xi_i \},\\
	&\mathcal{U}_i:= \{ S_1(0)=\xi_i,\,S_2(0)=2\xi_i,\,S_1(\xi_i)=2\xi_i \}.
\end{align*}
The independence of frogs yields that for any $i \in [1,d]$,
\begin{align*}
	P(\mathcal{S}_i)=P(\mathcal{T}_i)=\frac{2d-1}{(2d)^3},\quad
	P(\mathcal{U}_i)=\frac{1}{(2d)^3},
\end{align*}
and hence 
\begin{align*}
	\sum_{i=1}^d P(T(0,2\xi_i)=2)
	= \sum_{i=1}^d\{ P(\mathcal{S}_i)+P(\mathcal{T}_i)+P(\mathcal{U}_i) \}
	= \frac{d(4d-1)}{(2d)^3}.
\end{align*}
With these observations, one has
\begin{align*}
	\sum_{y \in \Delta_2} P(T(0,y)=2)
	\leq \frac{2d(d-1)}{(2d)^2}+\frac{d(4d-1)}{(2d)^3}
	= \half-\frac{1}{8d^2}<\half,
\end{align*}
and \eqref{eq:half} is proved.
\end{proof}

\begin{proof}[\bf Proof of Lemma~\ref{lem:q_induction}]
Fix $A \geq 2$.
We have for all $n \geq 2$,
\begin{align*}
	q_{An}=Q_1(n)+Q_2(n)+Q_3(n),
\end{align*}
where
\begin{align*}
	&Q_1(n):=E\Biggl[ \sum_{y \in R_{An-A-1}} \1{\{ T(0,y)+\tau(y,\Delta_{An})=An \}} \Biggr],\\
	&Q_2(n):=E\Biggl[ \sum_{y \in {\Delta_{An-A}}}\1{\{ T(0,y)+\tau(y,\Delta_{An})=An \}} \Biggr],\\
	&Q_3(n):=E\Biggl[ \sum_{y \in R_{An-1} \setminus R_{An-A}} \1{\{ T(0,y)+\tau(y,\Delta_{An})=An \}} \Biggr].
\end{align*}
Our goal is to prove that for all $n \geq 2$,
\begin{align}\label{eq:Q1Q2}
	Q_1(n)=q_{An-A}\times P^0(\tau(0,\Delta_A)=A) \geq Q_2(n)
\end{align}
and
\begin{align}\label{eq:Q3}
	Q_3(n) \leq 2q_{An-A} \times E\Biggl[ \sum_{x \in R_{A-1} \setminus \{ 0 \}}\1{\{ T(0,x)+\tau(x,\Delta_A)=A \}} \Biggr].
\end{align}
Indeed, \eqref{eq:Q1Q2} and \eqref{eq:Q3} lead to
\begin{align*}
	q_{An}
	&\leq 2q_{An-A} \Biggl( P^0(\tau(0,\Delta_A)=A)
		+E\Biggl[ \sum_{x \in R_{A-1} \setminus \{ 0 \}}\1{\{ T(0,x)+\tau(x,\Delta_A)=A \}} \Biggr] \Biggr)\\
	&= q_{An-A}(2q_A),
\end{align*}
and the lemma follows by induction on $n$.

We fix $n \geq 2$ and begin by proving the equality of \eqref{eq:Q1Q2}.
Note that $T(0,y)$ and $\tau(y,\Delta_{An})$ are independent for all $y \in \Z^d$.
This, together with the strong Markov property and translation invariance of frogs, shows that for any $y \in R_{An-A-1}$,
\begin{align*}
	&P\bigl( T(0,y)+\tau(y,\Delta_{An})=An \bigr)\\
	&= P(T(0,y)=\|y\|) P^y\bigl( \tau(y,\Delta_{An-A})=An-A-\|y\| \bigr) P^0(\tau(0,\Delta_A)=A)\\
	&= P\bigl( T(0,y)+\tau(y,\Delta_{An-A})=An-A \bigr) P^0(\tau(0,\Delta_A)=A).
\end{align*}
Hence,
\begin{align*}
	Q_1(n)
	&= \sum_{y \in R_{An-A-1}} P\bigl( T(0,y)+\tau(y,\Delta_{An-A})=An-A \bigr) P^0(\tau(0,\Delta_A)=A)\\
	&= q_{An-A}\times P^0(\tau(0,\Delta_A)=A),
\end{align*}
and the equality of \eqref{eq:Q1Q2} is valid.

Let us next check the inequality of \eqref{eq:Q1Q2}.
Similarly to the above, we obtain
\begin{align*}
	Q_2(n)
	&= \sum_{y \in \Delta_{An-A}} P(T(0,y)=\|y\|) P^y(\tau(y,\Delta_{An})=A)\\
	&= \sum_{y \in \Delta_{An-A}} P(T(0,y)=An-A)P^0(\tau(0,\Delta_A)=A).
\end{align*}
Since for each $y \in \Delta_{An-A}$,
\begin{align*}
	\{ T(0,y)=An-A \}=\bigcup_{x \in R_{An-A-1}} \{ T(0,x)+\tau(x,y)=An-A \},
\end{align*}
the union bound shows
\begin{align*}
	&\sum_{y \in \Delta_{An-A}} P(T(0,y)=An-A)\\
	&\leq \sum_{x \in R_{An-A-1}} \sum_{y \in \Delta_{An-A}} P(T(0,x)+\tau(x,y)=An-A)\\
	&= \sum_{x \in R_{An-A-1}} P\bigl( T(0,x)+\tau(x,\Delta_{An-A})=An-A \bigr)
		= q_{An-A}.
\end{align*}
Here, in the first equality, we used the fact that the events $\{ T(0,x)+\tau(x,y)=An-A \}$, $y \in \Delta_{An-A}$, are disjoint.
With these observations,
\begin{align*}
	Q_2(n)
	\leq q_{An-A} \times P^0(\tau(0,\Delta_A)=A),
\end{align*}
which establishes the inequality of \eqref{eq:Q1Q2}.

We finally prove \eqref{eq:Q3}.
For any $x \in R_{An-A}$ and $y \in R_{An-1} \setminus R_{An-A}$, consider the event
\begin{align*}
	\mathcal{E}_{x,y}:=\mleft\{
	\begin{minipage}{23em}
		$\tau(x,\Delta_{An-A})=An-A-\|x\|$ and\\
		$\min_{z \in R_{An-1} \setminus R_{An-A}}( \sigma_x(z)+\rho(z,y))+\tau(y,\Delta_{An})=A$
	\end{minipage}
	\mright\},
\end{align*}
where
\begin{align*}
	&\sigma_x(z):=\inf\bigl\{ k \geq 0:S_{\tau(x,\Delta_{An-A})+k}(x)=z \bigr\},\\
	&\rho(z,y):=\inf\mleft\{ \sum_{i=0}^{m-1}\tau(x_i,x_{i+1}):
		\begin{minipage}{18.7em}
			$m \geq 1$ and $x_0,x_1,\dots,x_m \in R_{An-1} \setminus R_{An-A}$\\ with $x_0=z$ and $x_m=y$
		\end{minipage}
	 \mright\},
\end{align*}
which are the time it takes for the frog sitting on $x$ to travel from $\Delta_{An-A}$ to $z$ and the first passage time from $z$ to $y$ which uses only frogs in $R_{An-1} \setminus R_{An-A}$, respectively.
Since $\mathcal{E}_{x,y}$ consists of frogs whose initial positions are in $\{x\} \cup (R_{An} \setminus R_{An-A})$, the events $\{ T(0,x)=\|x\| \}$ and $\mathcal{E}_{x,y}$ are independent.
This together with the union bound shows
\begin{align*}
	Q_3(n)
	\leq \sum_{x \in R_{An-A}} P(T(0,x)=\|x\|) \sum_{y \in R_{An-1} \setminus R_{An-A}} P(\mathcal{E}_{x,y}).
\end{align*}
Assume that for any $x \in R_{An-A}$,
\begin{align}\label{eq:Q3_key}
\begin{split}
	\sum_{y \in R_{An-1} \setminus R_{An-A}} P(\mathcal{E}_{x,y})
	&= P^x\bigl( \tau(x,\Delta_{An-A})=An-A-\|x\| \bigr)\\
	&\quad \times \sum_{y \in R_{A-1} \setminus \{0\}} P\bigl( T(0,y)+\tau(y,\Delta_A)=A \bigr).
\end{split}
\end{align}
Then, due to the independence of $T(0,x)$ and $\tau(x,\Delta_{An-A})$,
\begin{align}\label{eq:Q3_bound}
\begin{split}
	Q_3(n)
	&\leq \sum_{x \in R_{An-A}} P\bigl( T(0,x)+\tau(x,\Delta_{An-A})=An-A \bigr)\\
	&\qquad\qquad\quad \times E\Biggl[ \sum_{y \in R_{A-1} \setminus \{ 0 \}}\1{\{ T(0,y)+\tau(y,\Delta_A)=A \}} \Biggr].
\end{split}
\end{align}
We use the union bound to estimate the above sum as follows:
\begin{align*}
	&\sum_{x \in R_{An-A}} P\bigl( T(0,x)+\tau(x,\Delta_{An-A})=An-A \bigr)\\
	&= q_{An-A}+\sum_{x \in \Delta_{An-A}} P(T(0,x)=An-A)\\
	&\leq q_{An-A}+\sum_{y \in R_{An-A-1}}\sum_{x \in \Delta_{An-A}} P\bigl( T(0,y)+\tau(y,x)=An-A \bigr).
\end{align*}
Since the events $\{ T(0,y)+\tau(y,x)=An-A \}$, $x \in \Delta_{An-A}$, are disjoint, it holds that for all $y \in R_{An-A-1}$,
\begin{align*}
	\sum_{x \in \Delta_{An-A}} P\bigl( T(0,y)+\tau(y,x)=An-A \bigr)
	= P\bigl( T(0,y)+\tau(y,\Delta_{An-A})=An-A \bigr).
\end{align*}
Therefore,
\begin{align*}
	&\sum_{x \in R_{An-A}} P\bigl( T(0,x)+\tau(x,\Delta_{An-A})=An-A \bigr)\\
	&\leq q_{An-A}+\sum_{y \in R_{An-A-1}} P\bigl( T(0,y)+\tau(y,\Delta_{An-A})=An-A \bigr)= 2q_{An-A}.
\end{align*}
This combined with \eqref{eq:Q3_bound} implies that
\begin{align*}
	Q_3(n) \leq 2q_{An-A} \times E\Biggl[ \sum_{y \in R_{A-1} \setminus \{ 0 \}}\1{\{ T(0,y)+\tau(y,\Delta_A)=A \}} \Biggr],
\end{align*}
and \eqref{eq:Q3} follows.

It remains to check the validity of \eqref{eq:Q3_key}.
Note that both $\rho(z,y)$ and $\tau(y,\Delta_{An})$ do not depend on frogs in $R_{An-A} \setminus R_{An-A}$ whenever $y,z \in R_{An-1} \setminus R_{An-A}$.
We use the strong Markov property of $S_\cdot(x)$ to obtain for all $x \in R_{An-A}$ and $y \in R_{An-1} \setminus R_{An-A}$,
\begin{align*}
	P^x(\mathcal{E}_{x,y})
	= \sum_{w \in \Delta_{An-A}} &P^x\bigl( \tau(x,\Delta_{An-A})=An-A-\|x\|,\,S_{\tau(x,\Delta_{An-A})}(x)=w \bigr)\\
	&\times P^w\Bigl( \min_{z \in R_{An-1} \setminus R_{An-A}}(\tau(w,z)+\rho(z,y))+\tau(y,\Delta_{An})=A \Bigr).
\end{align*}
Hence, the independence of frogs yields
\begin{align}\label{eq:markov}
\begin{split}
	&\sum_{y \in R_{An-1} \setminus R_{An-A}} P(\mathcal{E}_{x,y})\\
	&= \sum_{w \in \Delta_{An-A}} P^x\bigl( \tau(x,\Delta_{An-A})=An-A-\|x\|,\,S_{\tau(x,\Delta_{An-A})}(x)=w \bigr)\\
	&\quad \times \sum_{y \in R_{An-1} \setminus R_{An-A}} P\Bigl( \min_{z \in R_{An-1} \setminus R_{An-A}}(\tau(w,z)+\rho(z,y))+\tau(y,\Delta_{An})=A \Bigr).
\end{split}
\end{align}
Furthermore, by the translation invariance of frogs, the last sum can be written as follows:
\begin{align*}
	&\sum_{y \in R_{An-1} \setminus R_{An-A}} P\Bigl( \min_{z \in R_{An-1} \setminus R_{An-A}}(\tau(w,z)+\rho(z,y))+\tau(y,\Delta_{An})=A \Bigr)\\
	&= \sum_{v \in R_{A-1} \setminus \{0\}} P\bigl( T(w,w+v)+\tau(w+v,w+\Delta_A)=A \bigr)\\
	&= \sum_{v \in R_{A-1} \setminus \{0\}} P\bigl( T(0,v)+\tau(v,\Delta_A)=A \bigr).
\end{align*}
Consequently, substituting this into the right side of \eqref{eq:markov} leads to \eqref{eq:Q3_key}.
\end{proof}

\subsection{Proof of Proposition~\ref{prop:delay}}\label{subsect:pf_delay}
The aim of this subsection is to show Proposition~\ref{prop:delay}.
To this end, define for any $L \in \N$ and $y \in \Z^d \setminus \{0\}$,
\begin{align*}
	&\mathcal{E}_1(L,y)
		:= \Biggl\{ T(0,y) \leq C_0\|y\|,\,\sum_{w \in \gamma_\mathcal{O}(y)} \1{\{ T(w,\mathcal{D}_L(w))=L \}}<\frac{\|y\|}{6L} \Biggr\},\\
	&\mathcal{E}_2(L,y)
		:= \Biggl\{ T(0,y) \leq C_0\|y\|,\,\sum_{w \in \gamma_\mathcal{O}(y) \setminus \{y\}} \tau(w,\overline{w}) \1{\{ \tau(w,\overline{w})>L \}}
		< \frac{\|y\|}{L^{d+3}} \Biggr\}.
\end{align*}
The key to proving Proposition~\ref{prop:delay} is to observe that if $L$ and $\|y\|$ are sufficiently large, then the events $\mathcal{E}_1(L,y)$ and $\mathcal{E}_2(L,y)$ have a negligible effect on delaying events.

We first estimate the probability of $\mathcal{E}_1(L,y)$ from above by using Propositions~\ref{prop:CN} and \ref{prop:Nakajima}.

\begin{lem}\label{lem:diamond}
There exist $\Cl[K]{phi:diamond} \in \mathcal{K}(d,r_0)$ and $L_1 \in \N$ (which depends only on $d$ and $r_0$) such that if $L \geq L_1$ and $\|y\| \geq \Cr{CN}(6L)^{d+2}$, then
\begin{align*}
	\sup_{r_0 \leq r \leq 1}\bP_r(\mathcal{E}_1(L,y)^c|\omega(0)=1) \leq \Cr{phi:diamond}(\|y\|).
\end{align*}
\end{lem}
\begin{proof}
Fix $r \in [r_0,1]$.
In addition, take $L_1 \in \N$ large enough to satisfy that for all $L \geq L_1$,
\begin{align*}
	e^{-\Cr{Nakajima}L} \leq C_0^{-d}\Cr{CN}^{-d}(6L+1)^{-d(d+2)}.
\end{align*}
We now consider the following family $(X_L(w))_{w \in \Z^d}$ of $2L$-dependent random variables and parameter $p_L$:
\begin{align*}
	X_L(w):=\1{\{ T(w,\mathcal{D}_L(w))=L \}},\qquad w \in \Z^d,
\end{align*}
and
\begin{align*}
	p_L:=\sup_{w \in \Z^d}\bP_r(X_L(w)=1)=\bP_r(X_L(0)=1).
\end{align*}
Lemma~\ref{lem:AP} implies that for all $y \in \Z^d \setminus \{0\}$,
\begin{align}\label{eq:diamond}
\begin{split}
	&\bP_r(\mathcal{E}_1(L,y)^c|\omega(0)=1)\\
	&\leq \phi_0(C_0\|y\|)
		+r_0^{-1}\bP_r\biggl( \max_{\pi \in \Pi_{C_0\|y\|}} \sum_{w \in \pi}X_L(w)
		\geq \frac{\|y\|}{6L} \biggr).
\end{split}
\end{align}
Note that the choice of $L_1$ and Proposition~\ref{prop:Nakajima} guarantee that if $L \geq L_1$ and $\|y\| \geq \Cr{CN}(6L)^{d+2}$, then
\begin{align*}
	p_L \leq e^{-\Cr{Nakajima}L} \leq (6L+1)^{-d}
\end{align*}
and
\begin{align*}
	\frac{\|y\|}{6L}
	\geq \Cr{CN}(2L)^{d+1}\max\{1,C_0\|y\|p_L^{1/d} \}.
\end{align*}
Hence, from Proposition~\ref{prop:CN}, there exists a constant $c$ (which depends only on $d$) such that
\begin{align*}
	\bP_r\biggl( \max_{\pi \in \Pi_{C_0\|y\|}} \sum_{w \in \pi}X_L(w) \geq \frac{\|y\|}{6L} \biggr)
	&\leq 2^d\exp\mleft\{ -\frac{\|y\|}{(32L)^{d+1}} \mright\}\\
	&\leq 2^d\exp\bigl\{ -c\|y\|^{1/(d+2)} \bigr\}.
\end{align*}
This together with \eqref{eq:diamond} implies that if $L \geq L_1$ and $\|y\| \geq \Cr{CN}(6L)^{d+2}$, then
\begin{align*}
	\bP_r(\mathcal{E}_1(L,y)^c|\omega(0)=1)
	\leq \phi_0(C_0\|y\|)+2^dr_0^{-1}\exp\bigl\{ -C\|y\|^{1/(d+2)} \bigr\}.
\end{align*}
It is clear that there exists $\Cr{phi:diamond} \in \mathcal{K}(d,r_0)$ such that for all $t \geq 0$,
\begin{align*}
	\phi_0(C_0t)+2^dr_0^{-1}\exp\bigl\{ -ct^{1/(d+2)} \bigr\}
	\leq \Cr{phi:diamond}(t),
\end{align*}
and the lemma follows.
\end{proof}

Let us next estimate the probability of $\mathcal{E}_2(L,y)$ from above by using Proposition~\ref{prop:CN} again.

\begin{lem}\label{lem:wandering}
There exist $\Cl[K]{phi:wandering} \in \mathcal{K}(d,r_0)$ and $L_2 \in \N$ (which depends only on $d$ and $r_0$) such that if $L \geq L_2$ and $\|y\|>(L+1)^{4d+10}$, then
\begin{align*}
	\sup_{r_0 \leq r \leq 1}\bP_r(\mathcal{E}_2(L,y)^c|\omega(0)=1) \leq \Cr{phi:wandering}(\|y\|).
\end{align*}
\end{lem}
\begin{proof}
Fix $r \in [r_0,1]$ and set $\delta:=1/(4d+10)$.
We define for any $L \in \N$ and $y \in \Z^d$ with $\|y\|^\delta \geq L+2$,
\begin{align*}
	\mathcal{V}(L,y)
	:=\Biggl\{ T(0,y) \leq C_0\|y\|,\,
	\sum_{w \in \gamma_\mathcal{O}(y) \setminus \{y\}} \tau(w,\bar{w})\1{\{ L<\tau(w,\bar{w})<\|y\|^\delta \}}
	\geq \frac{\|y\|}{L^{d+3}} \Biggr\}.
\end{align*}
From Lemmata~\ref{lem:AP} and \ref{lem:next} and the fact that $\tau(w,\bar{w})=T(w,\bar{w})$ when $w \in \gamma_\mathcal{O}(y) \setminus \{y\}$,
\begin{align}\label{eq:E2}
\begin{split}
	&\bP_r(\mathcal{E}_2(L,y)^c|\omega(0)=1)\\
	&\leq \phi_0(\|y\|)+(2C_0\|y\|+1)^{2d}\Cr{phi:next}(\|y\|^\delta)+\bP_r(\mathcal{V}(L,y)|\omega(0)=1).
\end{split}
\end{align}
To estimate the last probability, for each $\ell \in \N$, we consider the following family $(X_\ell(z))_{z \in \Z^d}$ of $2\ell$-dependent random variables and parameter $p_\ell$:
\begin{align*}
	X_\ell(w):=\1{\{ \exists w' \in B(w,\ell) \setminus \{w\} \text{ such that } T(w,w')=\tau(w,w')=\ell \}},\qquad w \in \Z^d,
\end{align*}
and
\begin{align*}
	p_\ell:=\sup_{w \in \Z^d}\bP_r(X_\ell(w)=1)=\bP_r(X_\ell(0)=1).
\end{align*}
Then, one has
\begin{align*}
	\sum_{w \in \gamma_\mathcal{O}(y) \setminus \{y\}} \tau(w,\bar{w})\1{\{ L<\tau(w,\bar{w})<\|y\|^\delta \}}
	&= \sum_{L<\ell<\|y\|^\delta} \ell \sum_{w \in \gamma_\mathcal{O}(y) \setminus \{y\}}\1{\{ \tau(w,\bar{w})=\ell \}}\\
	&\leq \sum_{L<\ell<\|y\|^\delta} \ell \sum_{w \in \gamma_\mathcal{O}(y) \setminus \{y\}} X_\ell(w).
\end{align*}
Hence, the same argument as in the proof of Lemma~\ref{lem:longwander} shows that for all $L \in \N$ and $y \in \Z^d$ with $\|y\|^\delta \geq L+2$,
\begin{align}\label{eq:E2'}
\begin{split}
	&\bP_r(\mathcal{V}(L,y)|\omega(0)=1)\\
	&\leq r_0^{-1}\sum_{L<\ell<\|y\|^\delta}
		\bP_r\biggl( \max_{\pi \in \Pi_{C_0\|y\|}} \sum_{w \in \pi} X_\ell(w) \geq \frac{d+3}{\ell^{d+5}}\|y\| \biggr).
\end{split}
\end{align}
We take $L_2 \in \N$ large enough to satisfy that $L_2 \geq 2^{d+1}\Cr{CN}$ and
\begin{align*}
	\Cr{phi:next}(t) \leq C_0^{-d}\Cr{CN}^{-d}(2t+1)^{-d(2d+7)},\qquad t \geq L_2.
\end{align*}
Lemma~\ref{lem:next} guarantees that if $L \geq L_2$ and $\|y\| \geq (L+2)^{1/\delta}$, then for all $\ell \in (L,\|y\|^\delta)$,
\begin{align*}
	p_\ell
	&\leq \sum_{w' \in B(0,\ell) \setminus \{0\}} \bP_r(T(0,w')=\tau(0,w')=\ell|\omega(0)=1)\\
	&\leq (2\ell+1)^d \Cr{phi:next}(\ell)
		\leq (6\ell+1)^{-d}
\end{align*}
and
\begin{align*}
	\frac{d+3}{\ell^{d+5}}\|y\| \geq \Cr{CN}(2\ell)^{d+1}\max\{1,C_0\|y\|p_\ell^{1/d} \}.
\end{align*}
This allows us to use Proposition~\ref{prop:CN}, and we have for some constant $c$ (which depends only on $d$),
\begin{align}\label{eq:E2'_LA}
\begin{split}
	&\bP_r\biggl( \max_{\pi \in \Pi_{C_0\|y\|}} \sum_{w \in \pi} X_\ell(w) \geq \frac{d+3}{\ell^{d+5}}\|y\| \biggr)\\
	&\leq 2^d\exp\mleft\{ -\frac{d+3}{32^d\ell^{2d+5}}\|y\| \mright\}
		\leq 2^d\exp\bigl\{ -c\|y\|^{1/2} \bigr\}.
\end{split}
\end{align}
Therefore, by \eqref{eq:E2}, \eqref{eq:E2'} and \eqref{eq:E2'_LA}, if $L \geq L_2$ and $\|y\| \geq (L+2)^{1/\delta}$, then
\begin{align*}
	&\bP_r(\mathcal{E}_2(L,y)^c|\omega(0)=1)\\
	&\leq \phi_0(\|y\|)+(2C_0\|y\|+1)^{2d}\Cr{phi:next}(\|y\|^\delta)+2^dr_0^{-1}\|y\|^\delta\exp\bigl\{ -c\|y\|^{1/2} \bigr\}.
\end{align*}
It is clear that there exists $\Cr{phi:wandering} \in \mathcal{K}(d,r_0)$ such that for all $t \geq 0$,
\begin{align*}
	\phi_0(t)+(2C_0t+1)^{2d}\Cr{phi:next}(t^\delta)+2^dr_0^{-1}t^\delta\exp\bigl\{ -ct^{1/2} \bigr\}
	\leq \Cr{phi:wandering}(t),
\end{align*}
and the lemma follows.
\end{proof}

We are now in a position to prove Proposition~\ref{prop:delay}.

\begin{proof}[\bf Proof of Proposition~\ref{prop:delay}]
Recall that for any $L \in \N$ and $y,z \in \Z^d \setminus \{0\}$,
\begin{align*}
	&\mathcal{E}(L,y,z)
		= \mleft\{ z \in \gamma_\mathcal{O}(y),\,T(z,\mathcal{D}_L(z))>L,\,\|\overline{z}^{(L)}-z\| \leq 2L \mright\},\\
	&\mathcal{E}_1(L,y)
		= \Biggl\{ T(0,y) \leq C_0\|y\|,\,\sum_{w \in \gamma_\mathcal{O}(y)} \1{\{ T(w,\mathcal{D}_L(w))=L \}}<\frac{\|y\|}{6L} \Biggr\},\\
	&\mathcal{E}_2(L,y)
		= \Biggl\{ T(0,y) \leq C_0\|y\|,\,\sum_{w \in \gamma_\mathcal{O}(y) \setminus \{y\}} \tau(w,\overline{w}) \1{\{ \tau(w,\overline{w})>L \}}
		< \frac{\|y\|}{L^{d+3}} \Biggr\},
\end{align*}
where $\overline{z}^{(L)}$ is the first point $w$ after $z$ along $\gamma_\mathcal{O}(y)$ satisfying that $\|w-z\|>L$ if such a point $w$ exists; otherwise set $\overline{z}^{(L)}:=z+3L\xi_1$ (see at the beginning of Section~\ref{subsect:influ_lower}).
Set $L:=\max\{ L_1,L_2,3^d\}$.
Assume that for all $y \in \Z^d$ with $\|y\| \geq \Cr{CN}(6L)^{4d+10}$ and $N \geq C_0\|y\|$, on the event $\mathcal{E}_1(L,y) \cap \mathcal{E}_2(L,y) \cap \{ \omega(0)=1 \}$,
\begin{align}\label{eq:lower_delay}
	\sum_{z \in B(0,N) \setminus \{0,y\}} \1{\mathcal{E}(L,y,z)}
	\geq \frac{\|y\|}{2L}.
\end{align}
Lemmata~\ref{lem:diamond} and \ref{lem:wandering} imply that if $\|y\|$ is sufficiently large and $N \geq C_0\|y\|$, then for any $r \in [r_0,1]$,
\begin{align*}
	&\sum_{z \in B(0,N) \setminus \{0,y\}} \bP_r(\mathcal{E}(L,y,z)|\omega(0)=1)\\
	&\geq \bE_r \Biggl[ \1{\mathcal{E}_1(L,y) \cap \mathcal{E}_2(L,y)}
		\sum_{z \in B(0,N) \setminus \{0,y\}} \1{\mathcal{E}(L,y,z)} \Bigg| \omega(0)=1 \Biggr]\\
	&\geq \frac{\|y\|}{2L}\{ 1-\Cr{phi:diamond}(\|y\|)-\Cr{phi:wandering}(\|y\|) \}
		\geq \frac{\|y\|}{4L},
\end{align*}
and the proposition follows.

It remains to show that for all $y \in \Z^d$ with $\|y\| \geq \Cr{CN}(6L)^{4d+10}$ and $N \geq C_0\|y\|$, \eqref{eq:lower_delay} holds on $\mathcal{E}_1(L,y) \cap \mathcal{E}_2(L,y) \cap \{ \omega(0)=1 \}$.
To this end, assume that $\mathcal{E}_1(L,y) \cap \mathcal{E}_2(L,y) \cap \{ \omega(0)=1 \}$ occurs.
We first estimate the number of frogs in $\gamma_\mathcal{O}(y)$. 
Since the sum of all $\tau(w,\bar{w})$ with $w \in \gamma_\mathcal{O}(y) \setminus \{y\}$ and $\tau(w,\bar{w})>L$ does not exceed $\|y\|/L^{d+3}$, one has
\begin{align*}
	\sum_{w \in \gamma_\mathcal{O}(y) \setminus \{y\}}\tau(w,\bar{w})\1{\{ \tau(w,\bar{w}) \leq L \}}
	\geq T(0,y)-\frac{\|y\|}{L^{d+3}}
	= \Bigl( 1-\frac{1}{L^{d+3}} \Bigr)\|y\|.
\end{align*}
This means that
\begin{align*}
	\#\{ w \in \gamma_\mathcal{O}(y) \setminus \{0,y\}:\tau(w,\bar{w}) \leq L \}
	\geq \frac{1}{L}\Bigl( 1-\frac{1}{L^{d+3}} \Bigr)\|y\|-1.
\end{align*}
In particular, we have
\begin{align}\label{eq:count1}
	\#(\gamma_\mathcal{O}(y) \setminus \{0,y\}) \geq \frac{1}{L}\Bigl( 1-\frac{1}{L^{d+3}} \Bigr)\|y\|-1.
\end{align}

Let us move on to the proof of \eqref{eq:lower_delay}.
The (random) subset $\Gamma(L,y)$ of $\Z^d$ is defined by
\begin{align*}
	\Gamma(L,y):=
	\mleft\{ z \in \gamma_\mathcal{O}(y) \setminus \{0,y\}:
		\begin{minipage}{18em}
			$y \not\in B(z,L)$ and $B(z,L)$ does not contain\\
			any $w \in \gamma_\mathcal{O}(y) \setminus \{y\}$ with $\tau(w,\bar{w})>L$
		\end{minipage}
		\mright\}.
\end{align*}
Use the fact that the sum of all $\tau(w,\bar{w})$ with $w \in \gamma_\mathcal{O}(y) \setminus \{y\}$ and $\tau(w,\bar{w})>L$ does not exceed $\|y\|/L^{d+3}$ again, and it holds that
\begin{align*}
	&\sum_{z \in \gamma_\mathcal{O}(y) \setminus \{y\}}
		\1{\{ \exists w \in \gamma_\mathcal{O}(y) \setminus \{y\} \text{ such that } w \in B(z,L) \text{ and } \tau(w,\bar{w})>L \}}\\
	&\leq \sum_{w \in \gamma_\mathcal{O}(y) \setminus \{y\}} \1{\{ \tau(w,\bar{w})>L \}}
		\sum_{z \in \gamma_\mathcal{O}(y) \setminus \{y\}} \1{\{ z \in B(w,L) \}}\\
	&\leq (2L+1)^d \times \frac{\|y\|}{L^{d+3}}
		\leq \frac{\|y\|}{L^2}.
\end{align*}
Hence, by \eqref{eq:count1},
\begin{align*}
	\#\Gamma(L,y)
	&\geq \frac{1}{L}\Bigl( 1-\frac{1}{L^{d+3}} \Bigr)\|y\|-2(2L+1)^d-\frac{\|y\|}{L^2}\\
	&\geq \frac{1}{L}\Bigl( 1-\frac{3}{L} \Bigr)\|y\|
		\geq \frac{2}{3L}\|y\|.
\end{align*}
This, together with the fact that $\|z-\bar{z}^{(L)}\| \leq 2L$ whenever $z \in \Gamma(L,y)$, shows that
\begin{align*}
	\sum_{z \in B(0,N) \setminus \{0,y\}} \1{\mathcal{E}(L,y,z)}
	&\geq \#\Gamma(L,y)-\#\bigl\{ z \in \gamma_\mathcal{O}(y):T(z,\mathcal{D}_L(z))=L \bigr\}\\
	&\geq \frac{2}{3L}\|y\|-\frac{\|y\|}{6L}=\frac{\|y\|}{2L},
\end{align*}
which implies \eqref{eq:lower_delay}.
\end{proof}


\section*{Acknowledgements}
V.~H.~Can is supported by the International Center for Research and Postgraduate Training in Mathematics, Institute of Mathematics, Vietnam Academy of Science and Technology under grant number ICRTM01\_2024.01.
N.~Kubota is supported by JSPS KAKENHI Grant Number JP20K14332.
S.~Nakajima is supported by JSPS KAKENHI Grant Number JP22K20344.



\end{document}